\documentclass[a4paper,12pt]{preprint}
\usepackage[margin=1.3in]{geometry}  
\usepackage[full]{textcomp}
\usepackage[osf]{newtxtext}

\usepackage{mathalfa}
\usepackage{microtype}

\usepackage{enumitem}
\usepackage{amsmath}
\usepackage{amsfonts} 
\usepackage{amssymb}             

\usepackage{comment}
\usepackage{breakurl}

\usepackage{mathrsfs}
\usepackage{booktabs}
\usepackage{mathtools}

\usepackage{tikz}
\usepackage{amsthm}                
\usepackage{mhequ}
\usepackage{hyperref}

\usepackage{accents}

\hypersetup{pdfstartview=}

\addtocontents{toc}{\protect\hypersetup{hidelinks}}

\DeclareSymbolFont{sfoperators}{OT1}{ptm}{m}{n}
\DeclareSymbolFontAlphabet{\mathsf}{sfoperators}

\makeatletter
\def\operator@font{\mathgroup\symsfoperators}
\makeatother

\numberwithin{equation}{section}

\newtheorem{thm}{Theorem}[section]
\newtheorem{defn}[thm]{Definition}
\newtheorem{lem}[thm]{Lemma}
\newtheorem{prop}[thm]{Proposition}

\newtheorem{assumption}[thm]{Assumption}
\theoremstyle{remark}

\newtheorem{rmk}[thm]{Remark}
\newtheorem{rem}[thm]{Remark}

\makeatletter
\def\th@newremark{\th@remark\thm@headfont{\bfseries}}

\def\bdiamond{\mathop{\mathpalette\bdi@mond\relax}}
\newcommand\bdi@mond[2]{%
	\vcenter{\hbox{\m@th
			\scalebox{\ifx#1\displaystyle 2.6\else1.8\fi}{$#1\diamond$}%
	}}%
}

\def\bDiamond{\mathop{\mathpalette\bDi@mond\relax}}
\newcommand\bDi@mond[2]{%
	\vcenter{\hbox{\m@th
			\scalebox{\ifx#1\displaystyle 2.6\else1.2\fi}{$#1\Diamond$}%
	}}%
}
\makeatletter

\definecolor{darkgreen}{rgb}{0.1,0.7,0.1}
\definecolor{darkred}{rgb}{0.7,0.1,0.1}
\definecolor{darkblue}{rgb}{0,0,0.7}
\newcommand{\RRR}{\mathbb{R}}

\newcommand{\EE}{\mathbb{E}}

\newcommand{\RR}{\mathbb{R}}

\newcommand{\fF}{\mathcal{F}}
\newcommand{\gG}{\mathcal{G}}
\newcommand{\hH}{\mathcal{H}}
\newcommand{\iI}{\mathcal{I}}

\newcommand{\nN}{\mathcal{N}}

\newcommand{\R}{\mathcal{R}}
\newcommand{\sS}{\mathcal{S}}

\newcommand{\xX}{\mathcal{X}}

\newcommand{\spacetime}{(I\times \RRR)}

\newcommand{\limn}{\lim_{n\rightarrow \infty}}
\newcommand{\limsupn}{\limsup_{n\rightarrow\infty}}

\newcommand{\eps}{\varepsilon}

\makeatletter 
\newcommand{\cov}{{\operator@font cov}}
\newcommand{\var}{{\operator@font var}}
\newcommand{\corr}{{\operator@font corr}}
\newcommand{\diam}{{\operator@font diam}}
\newcommand{\Av}{{\operator@font Av}}
\newcommand{\trig}{{\operator@font trig}}
\newcommand{\Enh}{{\operator@font Enh}}
\makeatother

\colorlet{symbols}{blue!90!black}
\colorlet{testcolor}{green!60!black}

\def\${|\!|\!|}

\makeatletter
\def\DeclareSymbol#1#2#3{\expandafter\gdef\csname MH@symb@#1\endcsname{\tikz[baseline=#2,scale=0.15,draw=symbols]{#3}}\expandafter\gdef\csname MH@symb@#1s\endcsname{\scalebox{0.7}{\tikz[baseline=#2,scale=0.15,draw=symbols]{#3}}}}
\def\<#1>{\csname MH@symb@#1\endcsname}
\makeatother

\setlist[itemize]{topsep=3pt,itemsep=1.5pt,parsep=0pt}

\def\scal#1{\langle#1\rangle}
\def\cent#1{\mathopen{{\langle\kern-0.3em\rangle}}#1\mathclose{{\langle\kern-0.3em\rangle}}}

\def\d{\partial}

\begin{document}

\title{Global well-posedness for the mass-critical stochastic nonlinear Schr\"odinger equation on $\RR$: general $L^2$ data}
\author{Chenjie Fan$^1$ and Weijun Xu$^2$}
\institute{University of Chicago, US, \email{cjfanpku@gmail.com}
\and University of Warwick, UK / NYU Shanghai, China, \email{weijunx@gmail.com}}

\maketitle

\begin{abstract}
We continue our study for the stochastic defocusing mass crtical nonlinear Schr\"odinger equation with conservative multiplicative noise, and show that it is globally well-posed for arbitrary initial data in $L_{\omega}^{\infty}L_{x}^{2}$. The main ingredients are several stability type results for deterministic (modified) NLS, which have their own interest. We also give some results on other stochastic NLS type models.
\end{abstract}

\setcounter{tocdepth}{2}
\microtypesetup{protrusion=false}
\tableofcontents
\microtypesetup{protrusion=true}

\section{Introduction}

\subsection{Main result}

The aim of the article is to establish the global well-posedness of the stochastic nonlinear Schr\"odinger equation
\begin{equation} \label{eq:stratonovich}
i \d_t u + \Delta u = |u|^4 u + u \circ \frac{{\rm d} W}{{\rm d} t}\;, \qquad x \in \RR\;, t \in \RR^+
\end{equation}
for arbitrary $L^2$-bounded initial data $u_0$ independent of the noise. The noise $\frac{{\rm d}W}{{\rm d} t}$ is a real-valued Gaussian processes that is white in time and coloured in space. The product $\circ$ appearing in the equation is in the Stratonovich sense, the only one that preserves the $L^2$ norm of the solution. 

We first give our precise assumption on the noise. Let $\hH$ be the Hilbert space of real-valued functions on $\RR$ with the inner product
\begin{equation*}
\scal{f,g}_{\hH} = \sum_{j=1}^{N} \scal{(1+|x|^K) f^{(j)}, (1+|x|^K) g^{(j)}}_{L^2}
\end{equation*}
for some sufficiently large $K$ and $N$ ($K, N = 10$ would be enough). Our assumption on the noise is the following. 

\begin{assumption} \label{as:noise}
	The noise $W$ has the form $W = \Phi \tilde{W}$, where $\Phi: L^2(\RR) \rightarrow \hH$ is a trace-class operator, and $\tilde{W}$ is the cylindrical Wiener process on $L^2(\RR)$. 
\end{assumption}

With this assumption on the noise, one can re-write \eqref{eq:stratonovich} in its It\^o form as
\begin{equation*}
i \d_t u + \Delta u = |u|^4 u + u \dot{W} - \frac{i}{2} u F_{\Phi}, 
\end{equation*}
where $\dot{W} = {\rm d} W / {\rm d} t$, the product between $u$ and $\dot{W}$ is in the It\^o sense, and
\begin{equation*}
F_{\Phi}(x) = \sum_{k} (\Phi e_k)^{2}(x)
\end{equation*}
is the It\^o-Stratonovich correction, which is independent of the choice of orthornormal basis $\{e_k\}$ of $L^2(\RR)$. 

We first introduce some notations. For every interval $\iI$, let
\begin{equation*}
\xX_1(\iI) = L_{t}^{\infty}L_{x}^{2}(\iI) = L^{\infty}(\iI, L^{2}(\RR))\;, \quad \xX_2(\iI) = L_{t}^{5}L_{x}^{10}(\iI) = L^{5}(\iI, L^{10}(\RR)), 
\end{equation*}
and $\xX = \xX_1 \cap \xX_2$ with $\|\cdot\|_{\xX(\iI)} = \|\cdot\|_{\xX_1(\iI)} + \|\cdot\|_{\xX_2(\iI)}$. For every $\rho \geq 1$, we also write $L_{\omega}^{\rho} \xX(\iI) = L^{\rho}(\Omega, \xX(\iI))$. Our main theorem is the following. 

\begin{thm} \label{th:main}
	Let $W$ be the Wiener process above and $u_0$ be independent of $W$ with $\|u_0\|_{L_{\omega}^{\infty}L_{x}^{2}} < +\infty$. Then, there exists a unique global flow $u$ adapted to the filtration generated by $W$ such that $u \in L_{\omega}^{\rho}\xX(0,T)$ for every $T>0$ and every $\rho \geq 5$, and satisfies
	\begin{equation} \label{eq:main_duhamel}
	\begin{split}
	u(t) &= \sS(t) u_0 - i \int_{0}^{t} \sS(t-s) \big( |u(s)|^{4} u(s) \big) {\rm d}s\\
	&-i \int_{0}^{t} \sS(t-s) u(s) {\rm d} W_s - \frac{1}{2} \int_{0}^{t} \sS(t-s) \big( F_{\Phi} u(s) \big) {\rm d}s, 
	\end{split}
	\end{equation}
	where the equality holds in $\xX(0,T)$ and the stochastic integral is in the It\^o sense. Furthermore, for every $\rho\geq 5$, we have
	\begin{equation*}
	\|u\|_{L_{\omega}^{\rho}\xX(0,T)} \lesssim_{\rho, \|u_0\|_{L_{\omega}^{\infty}L_{x}^{2}},T}  1.
	\end{equation*}
	and we have pathwise mass conservation in the sense that $\|u(t)\|_{L_x^2} = \|u_0\|_{L_x^2}$ almost surely for every $t \in [0,T]$. 
\end{thm}

\begin{rmk}
	In proving Theorem~\ref{th:main}, we will first choose a deterministic but sufficiently small $T_0$ depending on $\|u_0\|_{L_{\omega}^{\infty}L_{x}^{2}}$ and prove the existence of $u \in L_{\omega}^{\rho}\xX(0,T_0)$. Then pathwise mass conservation will allow us to extend this $u$ globally in time. 
\end{rmk}

\begin{rem}
Strictly pathwise conservation of mass in not necessary, as far as one can have pathwise control on the growth of the mass, same result will hold. For example, for the same Wiener process, if one consider equation
$i \d_t u + \Delta u = |u|^4 u + u \dot{W} $, (without Ito-Stratonicvich correction), since one has pathwise control on the growth of the mass, same result will hold. (In this example ,the mass grows at most exponentially in time.)
\end{rem}

\begin{rem}
The analysis in this paper, with slight modification, indeed gives the parallel results for the focusing case. Let  $u$ solves
\begin{equation*}
i \d_t u + \Delta u = -|u|^4 u + u \dot{W} - \frac{i}{2} u F_{\Phi}, 
\end{equation*}
with initial data $\|u_{0}\|_{L_{\omega}^{\infty}L_{x}^{2}}< \|Q\|_{L_{x}^{2}}$, then $u$ is global and for any $T>0$, one has
\begin{equation}
\|u\|_{L_{\omega}^{\rho}\xX(0,T)} \leq C_{\rho, \|u_0\|_{L_{\omega}^{\infty}L_{x}^{2}},T} .
\end{equation}
Here $Q$  is the unique postive $L^{2}$ solution solves
\begin{equation}
-\Delta Q+Q=|Q|^{4}Q.
\end{equation}
\end{rem}

Well-posedness of mass subcritical stochastic nonlinear Schr\"odinger equation with multiplicative noise was initiated in \cite{BD}, and has been extensively studied since then (see for example \cite{Rockner1}, \cite{Rockner2}). See also the recent article \cite{Hornung}.

\subsection{Construction of the solution and uniform boundedness}

The solution $u \in L_{\omega}^{\rho}\xX(0,T_0)$ in Theorem~\ref{th:main} was constructed when the initial data $u_0$ satisfies $\|u_0\|_{L_{\omega}^{\infty}L_{x}^{2}} \leq \delta_0$ for some sufficiently small $\delta_0$, and the main improvement of the current article is that we are able to drop the smallness assumption of the initial data. 

Let us give a brief overview on how the solution was constructed in \cite{snls_small}, and see where we used the smallness of initial data there.

We start with the existence of the solution in the truncated subcritical problem. Let $\theta: \RR^+ \rightarrow \RR^+$ be smooth with compact support in $[0,2)$, and $\theta=1$ on $[0,1]$. For every $m>0$, let
\begin{equation*}
\theta_{m}(x) = \theta(x/m). 
\end{equation*}
Also, for every $\eps>0$, we let $\nN^\eps(u) = |u|^{4-\eps} u$ and also $\nN(u) = |u|^4 u$. A result in \cite{BD} states that for every $m,\eps>0$ and $u_0 \in L_{\omega}^{\infty}L_{x}^{2}$, there is a unique process $u_{m,\eps} \in L_{\omega}^{\rho}\xX(0,T_0)$ satisfying
\begin{equation} \label{eq:duhamel_me}
\begin{split}
u_{m,\eps}(t) = &\sS(t) u_0 - i \int_{0}^{t} \sS(t-s) \Big( \theta_{m}\big( \|u_{m,\eps}\|_{\xX_2(0,s)} \big) \nN^{\eps} \big( u_{m,\eps}(s) \big) \Big) {\rm d}s\\
&- i \int_{0}^{t} \sS(t-s) u_{m,\eps}(s) {\rm d} W_s - \frac{1}{2} \int_{0}^{t} \sS(t-s) \big( F_{\Phi} u_{m,\eps}(s) \big) {\rm d}s, 
\end{split}
\end{equation}
and one has the pathwise mass conservation $\|u_{m,\eps}(t)\|_{L_{x}^{2}} = \|u_0\|_{L_x^2}$ almost surely. In \cite{snls_small}, starting from the family $\{u_{m,\eps}\}_{m,\eps}$, we took the following procedures: 
\begin{enumerate}
	\item For every $m>0$, we were able to show that the sequence $\{u_{m,\eps}\}_{\eps}$ converges in $L_{\omega}^{\rho}\xX(0,T_0)$ to a limit $u_m$, which satisfies
	\begin{equation} \label{eq:duhamel_m}
	\begin{split}
	u_{m}(t) = &\sS(t) u_0 - i \int_{0}^{t} \sS(t-s) \Big( \theta_{m}\big( \|u_{m}\|_{\xX_2(0,s)}^{5} \big) \nN \big( u_{m}(s) \big) \Big) {\rm d}s\\
	&- i \int_{0}^{t} \sS(t-s) u_{m}(s) {\rm d} W_s - \frac{1}{2} \int_{0}^{t} \sS(t-s) \big( F_{\Phi} u_{m}(s) \big) {\rm d}s
	\end{split}
	\end{equation}
	in the same space. This step requires only $\|u_0\|_{L_{\omega}^{\infty}L_{x}^{2}} < +\infty$ but no assumption on its actual value. 
	
	\item In the second step, starting from the sequence $\{u_m\}$ as obtained in the previous step, we were able to show that $u_m \rightarrow u$ in $L_{\omega}^{\rho}\xX(0,T_0)$ and the limit $u$ satisfies \eqref{eq:main_duhamel}. It is the proof of this convergence in \cite{snls_small} that used the smallness of $\|u_0\|_{L_{\omega}^{\rho}\xX(0,T_0)}$. More precisely, we showed that there exists $\delta_0$ sufficiently small such that if $\|u_0\|_{L_{\omega}^{\rho}\xX(0,T_0)} \leq \delta_0$, then the solution $u_{m,\eps}$ to \eqref{eq:duhamel_me} satisfies the bound
	\begin{equation*}
	\|u_{m,\eps}\|_{L_{\omega}^{\rho}\xX(0,T_0)} \leq B
	\end{equation*}
	for some $B$ independent of $m$ and $\eps$. It is this uniform bound that allows us to take the limit as $m \rightarrow +\infty$ and obtain convergence of $u_m$ to $u$. 
\end{enumerate}

In fact, the only place where we used the smallness of $\|u_0\|_{L_{\omega}^{\infty}L_{x}^{2}}$ is to establish the uniform boundedness of $\|u_{m,\eps}\|_{L_{\omega}^{\rho}L_{x}^{2}}$. As long as one has this uniform bound, one necessarily has the desired convergence $u_{m} \rightarrow u$. 

Hence, the key to construct the solution $u$ to \eqref{eq:main_duhamel} with arbitrary $u_0 \in L_{\omega}^{\infty}L_{x}^{2}$ is to prove the uniform boundedness of $\|u_{m,\eps}\|_{L_{\omega}^{\rho}\xX(0,T_0)}$. In fact, this is the key bound in our article. It is stated in the following theorem. 

\begin{thm} \label{th:snls_uni_bd}
	There exists $T_0 > 0$ such that for every $M>0$ and $\rho \geq 5$, there exists $B=B_{M,\rho}>0$ such that the solution $u_{m,\eps}$ to \eqref{eq:duhamel_me} with $\|u_0\|_{L_{\omega}^{\infty}L_{x}^{2}} \leq M$ satisfies
	\begin{equation*}
	\|u_{m,\eps}\|_{L_{\omega}^{\rho}\xX(0,T_0)} \leq B. 
	\end{equation*}
	The bound is uniform over $m>0$ and $\eps \in (0,1)$. 
\end{thm}

\subsection{Key ingredients}

The main ingredient in establishing the above uniform bound for $u_{m,\eps}$ is a uniform boundedness result for a deterministic equation in its integral form, stated in Theorem~\ref{th:main_bd} below. The proof of this theorem requires a series of boundedness and stability statements for perturbed nonlinear Sch\"odinger equation. These results are of their own interests, and we first state them below. Throughout, $\iI = [a,b]$ denotes an interval with length at most $1$, and all constants are independent of its end points $a$ and $b$ as long as $b-a \leq 1$.

\begin{prop} [Uniform stability]
	\label{pr:sta_me}
	Let $m, A, \tilde{A}>0$ and $\eps \in (0,1)$. Let $w, v \in \xX(\iI)$ and $e \in L_{t}^{1}L_{x}^{2}(\iI)$ such that
	\begin{equation*}
	i \d_t w + \Delta w = \theta_{m} \big( A + \|w\|_{\xX_2(a,t)}^{5} \big) \nN^\eps(w)\;, \quad w(a) \in L_x^2, 
	\end{equation*}
	and
	\begin{equation*}
	i \d_t v + \Delta v = \theta_{m} \big( \tilde{A} + \|v\|_{\xX_2(a,t)}^{5} \big) \nN^\eps(v) + e\;, \quad v(a) \in L_x^2, 
	\end{equation*}
	where both equations hold in $\iI$. Then, for every $M_1, M_2>0$, there exist $\delta, C>0$ such that if $\|v\|_{\xX_1(\iI)} \leq M_1$, $\|v\|_{\xX_2(\iI)} \leq M_2$, and
	\begin{equation*}
	\|v(a)-w(a)\|_{L_x^2} + \|e\|_{L_{t}^{1}L_{x}^{2}(\iI)} + |\tilde{A}-A| \leq \delta, 
	\end{equation*}
	then we have
	\begin{equation*}
	\|v-w\|_{\xX(\iI)} \leq C \big( \|v(a)-w(a)\|_{L_x^2} + \|e\|_{L_{t}^{1}L_{x}^{2}(\iI)} + |\tilde{A}-A| \big). 
	\end{equation*}
	Here, $\delta$ and $C$ depend on $M_1$ and $M_2$ only and in particular, they are independent of $m$, $A$, $\tilde{A}$ and $\eps$. 
\end{prop}

Unlike Proposition \ref{pr:sta_me} which is pertubrative, the following uniform bound is not purely pertubative.

\begin{prop} [Uniform boundedness]
	\label{pr:bd_me}
	Let $w$ be the solution to the equation
	\begin{equation} \label{eq:nls_w_me}
	i \d_t w + \Delta w = \theta_{m}\big(A+\|w\|_{\xX_{2}(a,t)}^{5}\big) \nN^{\eps}(w)\;, \quad w(a) \in L_{x}^{2}
	\end{equation}
	for some $m, A>0$ and $\eps \in (0,1)$. Then for every $M>0$, there exists $D_M>0$ such that
	\begin{equation*}
	\|w\|_{\xX(\iI)} \leq D_M
	\end{equation*}
	for all $w(a)$ satisfying $\|w(a)\|_{L_{x}^{2}} \leq M$. The bound is uniform in $m$, $A$ and $\eps$. 
\end{prop}

With uniform boundedness, we can enhance the uniform stability by dropping the assumption on $\|v\|_{\xX_2(\iI)}$. The statement is in the following proposition. 

\begin{prop} [Strong uniform stability]
	\label{pr:strong_sta_me}
	Suppose $v, e \in \xX(\iI)$ satisfy
	\begin{equation*}
	i \d_t v + \Delta v = \theta_{m} \big( \tilde{A} + \|v\|_{\xX_2(a,t)}^{5} \big) \nN^\eps(v) + e\;, \quad v_0 \in L_x^2. 
	\end{equation*}
	Let $w$ satisfy \eqref{eq:nls_w_me} with the same $m$ and $\eps$. Then for every $M>0$, there exist $\delta_M, C_M>0$ such that if $\|w(a)\|_{L_{x}^{2}} \leq M$, and
	\begin{equation*}
	\|v(a)-w(a)\|_{L_{x}^{2}} + \|e\|_{L_{t}^{1}L_{x}^{2}(\iI)} + |\tilde{A}-A| \leq \delta_M, 
	\end{equation*}
	then we have
	\begin{equation*}
	\|v-w\|_{\xX(\iI)} \leq C_M \big( \|v(a)-w(a)\|_{L_{x}^{2}} + \|e\|_{L_{t}^{1}L_{x}^{2}(\iI)} + |\tilde{A}-A| \big). 
	\end{equation*}
	The constants $\delta_M$ and $C_M$ depend on $M$, but are independent of $m$, $\eps$, $A$ and $\tilde{A}$. 
\end{prop}

With Propositions~\ref{pr:bd_me} and ~\ref{pr:strong_sta_me}, we can prove the following main deterministic theorem, which is the key ingredient to establish the uniform boundedness of $u_{m,\eps}$ in Theorem~\ref{th:snls_uni_bd}. 

\begin{thm} \label{th:main_bd}
	Suppose $u, g \in \xX(\iI)$ satisfy $g(a)=0$ and
	\begin{equation*}
	u(t) = e^{i(t-a)\Delta}u(a) - i \int_{a}^{t} \sS(t-s) \Big( \theta_{m}\big(A+\|u\|_{\xX_2(a,s)}^{5}\big) \nN^\eps\big(u(s)\big) \Big) {\rm d}s + g(t)
	\end{equation*}
	on $\iI=[a,b]$. Then, for every $M>0$, there exist $\eta_M, B_M>0$ such that whenever
	\begin{equation*}
	\|u\|_{\xX_1(\iI)} \leq M\;, \qquad \|g\|_{\xX_2(\iI)} \leq \eta_M\;, 
	\end{equation*}
	we have
	\begin{equation*}
	\|u\|_{\xX_2(\iI)} \leq B_M. 
	\end{equation*}
	The bound is uniform in $m$, $A$ and $\eps$. 
\end{thm}

\subsection{Structure of the article}

As previously explained, Theorem~\ref{th:main} essentially follows from Theorem~\ref{th:snls_uni_bd}, which in turn is a consequence of Theorem~\ref{th:main_bd}. 

The rest of the paper is organised as follows. In Section~\ref{sec: prelimi}, we give some preliminaries of diserpsive and Strichartz estimates, which will be used throughout the paper. In Section \ref{sec: proofofthmainbd}, we prove Theorem~\ref{th:main_bd} assuming Propositions~\ref{pr:bd_me} and~\ref{pr:strong_sta_me}. Then in Section~\ref{sec:snls_bd}, we prove Theorem~\ref{th:snls_uni_bd} as a consequence of Theorem~\ref{th:main_bd}. 

Sections~\ref{secroadmap1} to ~\ref{sectechmain} are devoted to the proof of Proposition~\ref{pr:bd_me}. In Section \ref{secroadmap1}, we present some stability results about the NLS, and reduce it to Proposition~\ref{pr:Dodson_eps}. In Section \ref{secroadmap2}, we present an overview for the proof Proposition~\ref{pr:Dodson_eps}, and reduce it to Proposition~\ref{propositon: techicnalmain}, which is then proved in the following Section~\ref{sectechmain}. 

Finally, the proof of Proposition~\ref{pr:sta_me} as well as the implication of Proposition~\ref{pr:strong_sta_me} from Propositions~\ref{pr:sta_me} and~\ref{pr:bd_me} are provided in the appendix.

\section{Preliminaries on Strichartz estimates and local theory of NLS}\label{sec: prelimi}

\subsection{Dispersive and Strichartz estimtes}

Recall the notation $\sS(t) = e^{it\Delta}$. We state some dispersive and Strichartz estimates below which are fundamental in the study of Schr\"odinger equation and are used throughout the article. These estimates are all stated in space dimension $d=1$. One may refer to \cite{cazenave2003semilinear} \cite{keel1998endpoint},\cite{tao2006nonlinear}  and reference therein.

\begin{prop} \label{pr:dispersive}
	There exists $C>0$ such that
	\begin{equation*}
	\|\sS(t) f\|_{L^{p'}(\RR)} \leq C t^{\frac{1}{p}-\frac{1}{2}} \|f\|_{L^{p}(\RR)}
	\end{equation*}
	for every $p \in [1,2]$ and every $f \in L^{p}(\RR)$. Here, $p'$ is the conjugate of $p$. 
\end{prop}

We now state Strichartz estimates. A pair of non-negative real numbers $(q,r)$ is called an admissible pair (for $d=1$) if
\begin{equation*}
\frac{2}{q} + \frac{1}{r} = \frac{1}{2}. 
\end{equation*}
We have the following Strichartz estimates. 

\begin{prop}
	For every admissible pair $(q,r)$, there exists $C>0$ depending on $(q,r)$ only such that 
	\begin{equation*}
	\|\sS(t) f\|_{L_{t}^{q}L_{x}^{r}(\RR)} \leq C \|f\|_{L_{x}^{2}}
	\end{equation*}
	for all $f \in L_{x}^{2}$. For every two admissible pairs $(q,r)$ and $(\tilde{q}, \tilde{r})$, there exists $C>0$ such that
	\begin{equation*}
	\Big\|\int_{\tau}^{t} \sS(t-s) f(s) {\rm d}s\Big\|_{L_{t}^{q}L_{x}^{r}(\iI)} \leq \|f\|_{L_{t}^{\tilde{q}'}L_{x}^{\tilde{r}'}(\iI)}
	\end{equation*}
	for all interval $\iI \subset \RR$ and all space-time functions $f \in L_{t}^{\tilde{q}'}L_{x}^{\tilde{r}'}(\iI)$. Here, $\tilde{q}'$ and $\tilde{r}'$ are the conjugates of $q$ and $r$. 
\end{prop}

\subsection{Preliminary for local theory}

It is now standard to prove local well posedness for mass critical/subcrtical NLS. We briefly review it for the convenience of the readers.. One may refer to \cite{cazenave1989some}, \cite{cazenave2003semilinear} and \cite{tao2006nonlinear} for more details.
We present the following Lemmas to summarize the key estimate in the local well-posedness. 

\begin{lem} \label{lem: prelwp}
Let $\iI=[a,b]$ and $0 \leq \eps \leq 1$. Let $u_{\eps}$ solve the equation
\begin{equation}
i\partial_{t}u_{\eps} + \Delta u_{\eps} = c\nN^\eps(u_\eps)\;, \qquad u_\eps(a) \in L_{x}^{2}. 
\end{equation}
Then 
\begin{equation}\label{eq: prelwp}
\|u_{\eps} - e^{ i t \Delta}u_{\eps}(a)\| \lesssim  c(b-a)^{\frac{\eps}{4}} \|u_\eps\|_{\xX_{1}(\iI)} \|u_\eps\|_{\xX_{2}(\iI)}^{4-\eps}
\end{equation}
for all $\eps \in [0,1]$. 
\end{lem}

\begin{lem}\label{lem: smallmass}
Let $\iI=[a,b]$. There exists $ \delta_{0}>0$ and $C>0$ depending on $b-a$ only such thats if $u_{\eps}$ solves 
\begin{equation}
i \d_{t} u_{\eps} + \Delta u_{\eps} = \nN^\eps(u_\eps)
\end{equation}
with $\eps \in [0,1]$ and $\|u_{\eps}(0)\|_{L_{x}^{2}} \leq \delta_{0}$, then
\begin{equation}
\sup_{\eps \in [0,1]} \|u_{\eps}\|_{\xX_2(\iI)} \leq C. 
\end{equation}
Both $\delta_0$ and $C$ are independent of $\eps \in [0,1]$. 
\end{lem}

In practice, we need to slightly enhance  Lemma \ref{lem: smallmass} to the following. 

\begin{lem} \label{lem: enhancedsmallness}
Let $\iI=[a,b]$ and $0<\eps<1$. There exist $\delta_{0}>0$ and $C>0$ depending on $b-a$ only such that if $u_{\eps}$ solves 
\begin{equation}
i \d_{t} u_{\eps} + \Delta u_{\eps} = \nN^\eps(u_\eps)
\end{equation}
with
\begin{equation}
\|e^{i t \Delta}u_{\eps}(a)\|_{\xX_2(\iI)} \leq \delta_{0}, 
\end{equation}
then
\begin{equation}
\sup_{\eps \in [0,1]} \|u_{\eps}\|_{\xX_2(\iI)} \leq C \|e^{i t \Delta} u_{\eps}(a)\|_{\xX_2(\iI)}. 
\end{equation}
Both $\delta_0$ and $C$ are uniform in $\eps \in [0,1]$. 
\end{lem}

\section{Proof of Theorem~\ref{th:main_bd}}\label{sec: proofofthmainbd}

We now prove Theorem~\ref{th:main_bd} assuming Propositions~\ref{pr:bd_me} and \ref{pr:strong_sta_me}. Let $v=u-g$. Since $g(a)=0$, $v$ satisfies
\begin{equation*}
v(t) = e^{i(t-a)\Delta}v(a) - \int_{a}^{t} \sS(t-s) \Big( \theta_{m}\big(A+\|v+g\|_{\xX_2(0,t)}^{5}\big) \nN^{\eps} \big(v(s)+g(s)\big) \Big) {\rm d}s. 
\end{equation*}
We re-write it in its differential form as
\begin{equation*}
i \d_t v + \Delta v = \theta_{m}\big(A+\|v\|_{\xX_2(0,t)}^{5}\big) + e\;, 
\end{equation*}
where the error term $e$ is given by
\begin{equation*}
e(t) = \theta_{m} \big( A + \|v+g\|_{\xX_2(0,t)}^{5} \big) \nN^{\eps}\big( v(t) + g(t) \big) - \theta_{m} \big( A + \|v\|_{\xX_2(0,t)}^{5} \big) \nN^{\eps}\big( v(t) \big). 
\end{equation*}
We are now in the form of Proposition~\ref{pr:strong_sta_me}. Let $w$ be the solution to \eqref{eq:nls_w_me} with initial data $w(a) = u(a)$ and the same $A$ as above. In order to establish the bound for $v$, we need to show that $e$ is small and hence we can apply Proposition~\ref{pr:strong_sta_me} to control $\|v-w\|_{\xX(\iI)}$. To do this, we split $e$ into two parts $e = e_1 + e_2$, where
\begin{equation*}
\begin{split}
e_1(t) &= \theta_{m}\big(A + \|v+g\|_{\xX_2(0,t)}^{5}\big) \Big( \nN^{\eps}\big(v(t) + g(t)\big) - \nN^\eps\big(v(t)\big) \Big); \\
e_2(t) &= \Big( \theta_{m}\big(A + \|v+g\|_{\xX_2(0,t)}^{5}\big) - \theta_{m}\big(A + \|v\|_{\xX_2(0,t)}^{5}\big) \Big) \nN^{\eps}\big(v(t)\big). 
\end{split}
\end{equation*}
They can be controlled pointwise by
\begin{equation*}
|e_1| \leq C |g| \big(|v|^{4-\eps} + |g|^{4-\eps}\big), \quad |e_2| \leq C |v|^{5-\eps}  \|g\|_{\xX_2(a,t)} \big( \|v\|_{\xX_{2}(0,t)}^{4} + \|g\|_{\xX_{2}(0,t)}^{4} \big).  
\end{equation*}
By H\"older's inequality, we get the bounds
\begin{equation*}
\begin{split}
\|e_1\|_{L_{t}^{1}L_{x}^{2}(a,r)} &\leq C \|g\|_{\xX_1(a,r)}^{\frac{\eps}{4}} \|g\|_{\xX_2(a,r)}^{1-\frac{\eps}{4}} \big( \|v\|_{\xX_2(a,r)}^{4-\eps} + \|g\|_{\xX_2(a,r)}^{4-\eps} \big); \\
\|e_2\|_{L_{t}^{1}L_{x}^{2}(a,r)} &\leq C \|g\|_{\xX_2(a,r)} \|v\|_{\xX_1(a,r)}^{\frac{\eps}{4}} \|v\|_{\xX_2(a,r)}^{5-\frac{5 \eps}{4}} \big( \|v\|_{\xX_2(a,r)}^{4} + \|g\|_{\xX_2(a,r)}^{4} \big). 
\end{split}
\end{equation*}
Combining the above two bounds together, and relaxing $\|v\|_{\xX_1}$ and $\|v\|_{\xX_2}$ to $\|v\|_{\xX}$, we deduce that there exists $C_0>0$ universal such that
\begin{equation} \label{eq:main_bootstrap}
\begin{split}
\|e\|_{L_{t}^{1}L_{x}^{2}(a,r)} \leq &C_0 \|g\|_{\xX_1(a,r)}^{\frac{\eps}{4}} \|g\|_{\xX_2(a,r)}^{1-\frac{\eps}{4}} \big( \|v\|_{\xX(a,r)}^{4-\eps} + \|g\|_{\xX_2(a,r)}^{4-\eps} \big)\\
&+ C_0 \|g\|_{\xX_2(a,r)} \|v\|_{\xX(a,r)}^{5-\eps} \big( \|v\|_{\xX(a,r)}^{4} + \|g\|_{\xX_2(a,r)}^{4} \big)
\end{split}
\end{equation}
for all $r \in [a,b]$. 

Let $D_M$, $\delta_M$ and $C_M$ be the constants in Propositions~\ref{pr:bd_me} and ~\ref{pr:strong_sta_me}. We claim that there exists $\eta>0$ depending on $M$ only such that if $\|g\|_{\xX_2(\iI)} \leq \eta$, then
\begin{equation} \label{eq:u_endpt}
\|v\|_{\xX(\iI)} + \|g\|_{\xX(\iI)} \leq 9(M + D_M). 
\end{equation}
This would immediately imply the desired bound for $u = v+g$ and conclude the proof of the Theorem~\ref{th:main_bd}. 

To see the existence of such an $\eta$ and the validity of \eqref{eq:u_endpt}, we first note that $\|v\|_{\xX(a,a)} + \|g\|_{\xX(a,a)} = M$ since $g(a)=0$. Suppose $r \in [a,b]$ is such that
\begin{equation} \label{eq:main_hypo}
\|v\|_{\xX(a,r)} + \|g\|_{\xX(a,r)} \leq 9 (M+D_M). 
\end{equation}
By \eqref{eq:main_bootstrap}, we know there exists $\eta>0$ depending on $M$ only such that
\begin{equation} \label{eq:main_choice_eta}
\|e\|_{L_{t}^{1}L_{x}^{2}(a,r)} \leq \delta_M \quad \text{and} \quad C_{M} \|e\|_{L_{t}^{1}L_{x}^{2}(a,r)} \leq M+D_M
\end{equation}
whenever $r$ satisfies \eqref{eq:main_hypo} and $\|g\|_{\xX_2(a,r)} \leq \|g\|_{\xX_2(\iI)} \leq \eta$. We can then use Proposition~\ref{pr:strong_sta_me} to deduce that
\begin{equation*}
\|v-w\|_{\xX(a,r)} \leq C_M \|e\|_{L_{t}^{1}L_{x}^{2}(a,r)} \leq M + D_M, 
\end{equation*}
which, combined with the uniform bound of $w$ in Proposition~\ref{pr:bd_me}, implies
\begin{equation} \label{eq:main_u_v}
\|v\|_{\xX(a,r)} \leq \|w\|_{\xX(a,r)} + \|v-w\|_{\xX(a,r)} \leq 2(M+D_M).  
\end{equation}
Returning to $u$, if we further require $\eta < M+D_M$ (which does not change anything above), we will have
\begin{equation*}
\|u\|_{\xX_2(a,r)} \leq \|v\|_{\xX_2(a,r)} + \|g\|_{\xX_2(a,r)} \leq 2(M+D_M) + \eta \leq 3(M+D_M), 
\end{equation*}
and hence
\begin{equation} \label{eq:main_u}
\|u\|_{\xX(a,r)} = \|u\|_{\xX_1(a,r)} + \|u\|_{\xX_2(a,r)} \leq 4(M+D_M), 
\end{equation}
where we used the assumption $\|u\|_{\xX_1(\iI)} \leq M$. Combining \eqref{eq:main_u_v} and \eqref{eq:main_u}, we have
\begin{equation*}
\|g\|_{\xX(a,r)} \leq \|v\|_{\xX(a,r)} + \|u\|_{\xX(a,r)} \leq 6(M+D_M). 
\end{equation*}
Using \eqref{eq:main_u_v} again, we deduce that
\begin{equation} \label{eq:main_improvement}
\|v\|_{\xX(a,r)} + \|g\|_{\xX(a,r)} \leq 8(M+D_M). 
\end{equation}
To summarize, we have shown that if $\eta$ is chosen according to \eqref{eq:main_choice_eta} and $\|g\|_{\xX_2(\iI)} \leq \eta$, then \eqref{eq:main_improvement} holds whenever \eqref{eq:main_hypo} is true. Since $\|v\|_{\xX(a,a)} + \|g\|_{\xX(a,a)} = M < 9(M+D_M)$, and the norm is continuous in $r \in [a,b]$, we can conclude that \eqref{eq:main_improvement} is true for all $r \in [a,b]$, and hence we obtain \eqref{eq:u_endpt}. This completes the proof of Theorem~\ref{th:main_bd}.

\section{Uniform boundedness of $u_{m,\eps}$: Theorem~\ref{th:main_bd} implies Theorem~\ref{th:snls_uni_bd}}
\label{sec:snls_bd}

Let $u_{m,\eps}$ satisfy \eqref{eq:duhamel_me} with $\|u_0\|_{L_{\omega}^{\infty}L_{x}^{2}} \leq M$. By pathwise mass conservation, we have $\|u_{m,\eps}\|_{\xX_1(0,T)} \leq M$ almost surely for every $T>0$. It then suffices to consider $\xX_2$ only. 

Let $\eta=\eta_M$ be as in Proposition~\ref{pr:strong_sta_me}. We first choose $T_0 \leq 1$ (depending on $M$ only) such that
\begin{equation*}
\frac{1}{2} \Big\| \int_{r_1}^{r_2} \sS(t-s) \big( F_{\Phi} u_{m,\eps}(s) \big) {\rm d}s \Big\|_{\xX_{2}(0,T_0)} \leq C T_{0}^{\frac{4}{5}} \|u_{m,\eps}\|_{\xX_1(0,T_0)} \leq \frac{\eta}{2}
\end{equation*}
for all $0 \leq r_1 \leq r_2 \leq T_0$. Now, fix this $T_0$ and let $\iI = [0,T_0]$. For $t \in \iI$, let
\begin{equation*}
M^*(t) := \sup_{0 \leq r_1 \leq r_2 \leq t} \Big\| \int_{r_1}^{r_2} \sS(t-s) u_{m,\eps}(s) {\rm d} W_s \Big\|_{L_{x}^{10}}. 
\end{equation*}
We have the following proposition controlling $M^*$. 

\begin{prop} \label{pr:bound_maximal}
	There exists $C>0$ depending on $\rho$ only such that
	\begin{equation} \label{eq:bound_maximal}
	\|M^*\|_{L_{\omega}^{\rho}L_{t}^{5}(0,T_0)} \leq C_{\rho} T_{0}^{\frac{3}{10}} \|u_{m,\eps}\|_{L_{\omega}^{\rho} \xX_1(0,T_0)} \leq C_{\rho} M T_{0}^{\frac{3}{10}}. 
	\end{equation}
\end{prop}
\begin{proof}
	See appendix. 
\end{proof}

Now we proceed with the proof of Theorem~\ref{th:main_bd}. Proposition~\ref{pr:bound_maximal} in particular implies that $\|M^*\|_{L_{t}^{5}(0,T_0)}$ is finite. Hence, we can choose a random dissection
\begin{equation*}
0 = \tau_0 < \tau_1 < \cdots < \tau_K = T_0
\end{equation*}
of $\iI$ in the following way. Let $\tau_0 = 0$. Suppose $\tau_k$ is chosen, and let
\begin{equation*}
\tau_{k+1} = T_0 \wedge \inf\Big\{\tau>\tau_k: \int_{\tau_k}^{\tau} |M^*(t)|^{5} {\rm d}t = \frac{\eta}{2} \Big\}, 
\end{equation*}
where we recall $\eta=\eta_M$ is specified at the beginning of the section. The total number $K$ of the intervals is bounded by
\begin{equation*}
K \leq 1 + \frac{2}{\eta} \cdot \|M^*\|_{L_{t}^{5}(\iI)}^{5}. 
\end{equation*}
Now fix $\omega \in \Omega$ and $k \leq K$ arbitrary. For $t \in [\tau_k, \tau_{k+1}]$, let
\begin{equation*}
g_{m,\eps}(t) = - i \int_{\tau_j}^{t} \sS(t-s) u_{m,\eps}(s) {\rm d} W_s - \frac{1}{2} \int_{\tau_j}^{t} \sS(t-s) \big( F_{\Phi} u_{m,\eps}(s) \big) {\rm d}s. 
\end{equation*}
Then, $g_{m,\eps}$ satisfies the assumption of Theorem~\ref{th:main_bd} on $[\tau_k, \tau_{k+1}]$, so we have
\begin{equation*}
\int_{\tau_j}^{\tau_{j+1}} \|u_{m,\eps}(t)\|_{L_{x}^{10}}^{5} {\rm d} t \leq B
\end{equation*}
for some $B$ depending on $M$ only. Since this is true for all $0 \leq k \leq K-1$, we can sum the above inequality over the intervals $[\tau_k, \tau_{k+1}]$ for all $k$ so that
\begin{equation*}
\int_{0}^{T_0} \|u_{m,\eps}(t)\|_{L_{x}^{10}}^{5} {\rm d} t \leq B \Big(1 + \frac{2}{\eta} \cdot \|M^*\|_{L_{t}^{5}(\iI)}^{5} \Big). 
\end{equation*}
Taking $5$-th root and then $L_{\omega}^{\rho}$-norm on both sides, and applying \eqref{eq:bound_maximal}, we obtain
\begin{equation*}
\|u_{m,\eps}\|_{L_{\omega}^{\rho}\xX(0,T_0)} \leq C B^{\frac{1}{5}} \big( 1 + \eta^{-\frac{1}{5}} \|M^{*}\|_{L_{\omega}^{\rho}L_{t}^{5}(0,T_0)} \big) \leq C B^{\frac{1}{5}} \big( 1 + C_{\rho} M T_{0}^{\frac{3}{10}} \big). 
\end{equation*}
Since both $B$ and $\eta$ depend on $M$ only, we can conclude the proof of the theorem.

\section{Roadmap to the proof of Propositions~\ref{pr:bd_me}}\label{secroadmap1}

\subsection{A brief review of stability for the mass critical NLS}

Fix $\iI = [a,b]$ with $b-a \leq 1$. Let $c \in [0,1]$, and $w \in \xX(\iI)$ be the solution to
\begin{equation} \label{eq:nls_w}
i \d_t w + \Delta w = c |w|^4 w\;, \qquad w(a) \in L_{x}^{2}, 
\end{equation}
and $v \in \xX(\iI)$ and $e \in L_{t}^{1}L_{x}^{2}(\iI)$ such that
\begin{equation} \label{eq:nls_v}
i \d_t v + \Delta v = c |v|^4 v +e\;, \quad v(a) \in L_{x}^{2}. 
\end{equation}
We call \eqref{eq:nls_w} mass critical NLS with parameter $c$. The following stability result is well known and purely pertubative. 

\begin{prop} \label{pr:iteam_sta}
Let $c \in [0,1]$ and $w \in \xX(\iI)$ be the solution to \eqref{eq:nls_w}. Let $v, e$ satisfy \eqref{eq:nls_v}. Then for every $M_1, M_2>0$, there exist $\delta = \delta_{M_1, M_2}$ and $C = C_{M_1, M_2}$ such that if
\begin{equation}
\|u\|_{\xX_1(\iI)} \leq M_{1}\;, \quad \|u\|_{\xX_2(\iI)} \leq M_{2}\;, \quad \|v(a)-w(a)\|_{L_{x}^{2}} +  \|e\|_{L_{t}^{1}L_{x}^{2}([a,b]\times \RRR)}\leq \delta, 
\end{equation}
then we have
\begin{equation}
\|u-w\|_{\xX(\iI)} \leq C \big( \|v(a)-w(a)\|_{L_{x}^{2}} +  \|e\|_{L_{t}^{1}L_{x}^{2}([a,b]\times \RRR)} \big). 
\end{equation}
In particular, $C_{M_{1},M_{2}}$ does not depend on $c$, and its dependence on $\iI$ is through $b-a$ only (and hence universal here since we assume $b-a \leq 1$). 
\end{prop}

The details of the above proposition can be found in Lemmas 3.9 and 3.10 in \cite{colliander2008global}. Apriori, it is highly nontrival that $w$ can be defined on all of $\iI$. However, the proposition indeed implies $w$ is well defined on $[a.b]$. It is proved by Dodson (\cite{dodson2012global, dodson2016global, dodson2467global}) that the solution to \eqref{eq:nls_w} with arbitrary $L^2$ initial data is global and scatters.

\begin{thm}[Dodson Scattering] \label{th:dodson}
	
Let $u$ solves NLS with initial data $u_{0}\in L_{x}^{2}$,
\begin{equation}
\begin{cases}
i \d_t u + \Delta u = c|u|^{4}u,\\
u_{0} \in L_{x}^{2}, 
\end{cases}
\end{equation}
and $0<c\leq 1$. Then $u$ is global and 
\begin{equation} \label{eq:dodson_bd}
\|u\|_{\xX_2(\RR)} \leq C=C_{\|u_{0}\|_{2}}.
\end{equation}
The bound does not depend on $c$. In particular, it implies that if $w$ solves \eqref{eq:nls_w}, then
\begin{equation*}
\|w\|_{\xX(\iI)} \lesssim_{\|w(a)\|_{L_x^2}} 1. 
\end{equation*}
\end{thm}

\begin{rem}
This theorem is highly nontrivial and completely non-pertubative.
\end{rem}

\begin{rem}
Theorem \ref{th:dodson} is usually stated for $c=1$. Clearly, after by multiplying a constant to the solution to \eqref{eq:nls_w}, one  derives Theorem \ref{th:dodson} for all $\delta<c\leq 1$ with a constant depending on $\delta$ from the case $c=1$. However, fixing $M=\|u_{0}\|_{L_{x}^{2}}$, when $c$ is small enough, the problem follows into the perturbation  scheme, in the spirit of Lemma \ref{lem: prelwp}. Thus, one has a uniform bound independent of $c$ in  \eqref{eq:dodson_bd}.
\end{rem}

With Theorem \ref{th:dodson}, one could enhance the stability result \ref{pr:iteam_sta} to the following. 

\begin{prop}\label{pr:strong_sta}
	Let $w$ be the solution to \eqref{eq:nls_w}, and $v,e$ satisfy \eqref{eq:nls_v}. For every $M>0$, there exist $\delta, C>0$ depending on $M$ only such that if
	\begin{equation*}
	\|w(a)\|_{L_{x}^{2}} \leq M\;, \quad \|v(a)-w(a)\|_{L_{x}^{2}} + \|e\|_{L_{t}^{1}L_{x}^{2}(\iI)} \leq \delta, 
	\end{equation*}
	then we have
	\begin{equation*}
	\|v-w\|_{\xX(\iI)} \leq C \big( \|v(a)-w(a)\|_{L_{x}^{2}} + \|e\|_{L_{t}^{1}L_{x}^{2}(\iI)} \big). 
	\end{equation*}
	The constants $\delta$ and $C$ depend on $M$ only. 
\end{prop}

Proposition \ref{pr:strong_sta} can be derived from Proposition \ref{pr:iteam_sta} and Proposition \ref{th:dodson} via a bootrap argument. Rather than give a proof here, we will indeed show how to use a bootstrap argument to derive Proposition~\ref{pr:strong_sta_me} from Proposition~\ref{pr:bd_me}, Propostion ~\ref{pr:sta_me}. The argument is in principle same

We remark here the key enhancement in Proposition \ref{pr:strong_sta} is that one does not need the control of $\|u\|_{\xX_2}$ any more. This proposition, depending on Theorem~\ref{th:dodson}, is of nonperturbative nature.

The Proposition \ref{pr:sta_me}, \ref{pr:bd_me}, \ref{pr:strong_sta_me} are natural generalizations of Proposition \ref{pr:iteam_sta}, Theorem \ref{th:dodson} and Proposition \ref{pr:strong_sta}, which are special situations when $m=+\infty$ and $\eps=0$. We will give a proof of Proposition~\ref{pr:sta_me} in the appendix, which is essentially the same as the proof of Proposition~\ref{pr:iteam_sta}.

\subsection{Parallel statements for $m=+\infty$}

We now present the parallel statements of Propositions~\ref{pr:sta_me}, \ref{pr:bd_me}, and~\ref{pr:strong_sta_me} for $m=\infty$ and $\eps \in [0,1]$. 

We fix the interval $\iI = [a,b]$, and let $c \in [0,1]$. Let $w_\eps \in \xX(\iI)$ such that
\begin{equation} \label{eq:nls_w_eps}
i \d_t w_\eps + \Delta w_\eps = c \nN^{\eps}(w_\eps)\;, \quad w_\eps(a) \in L_x^2\;, 
\end{equation}
and $v_\eps \in \xX(\iI)$, $e_\eps \in L_{t}^{1}L_{x}^{2}(\iI)$ such that
\begin{equation} \label{eq:nls_v_eps}
i \d_t v_\eps + \Delta v_\eps = c \nN^\eps(v_\eps) + e_\eps\;, \quad v_\eps(a) \in L_x^2. 
\end{equation}
Throughout this section, $w_\eps$, $v_\eps$ and $e_\eps$ satisfy the above equations. All constants below depend on the interval $\iI$ through $b-a$ only (and is uniform for all intervals with smaller lengths).

\begin{prop} [Uniform-in-$\eps$ stability]
	\label{pr:sta_eps}
	Let $w_\eps$ be the solution to \eqref{eq:nls_w_eps}, and $v_\eps, e_\eps$ satisfies \eqref{eq:nls_v_eps}. Then, for every $M_{1}, M_2>0$, there exist $\delta, C>0$ depending on $M_1$ and $M_2$ only such that if $\|v_\eps\|_{\xX_1(\iI)} \leq M_1$, $\|v_\eps\|_{\xX_2(\iI)} \leq M_2$ and
	\begin{equation*}
	\|v_\eps(a) - w_\eps(a)\|_{L_{x}^{2}} + \|e\|_{L_{t}^{1}L_{x}^{2}} \leq \delta, 
	\end{equation*}
	then we have
	\begin{equation*}
	\|v_\eps - w_\eps\|_{\xX(\iI)} \leq C \big( \|v_\eps(a) - w_\eps(a)\|_{L_{x}^{2}} + \|e\|_{L_{t}^{1}L_{x}^{2}} \big). 
	\end{equation*}
	In fact, the constants $\delta$ and $C$ can depend on $M_2$ only. 
\end{prop}
The proof of Proposition \ref{pr:sta_eps} is almost exactly same as the proof of Proposition \ref{pr:iteam_sta} and the proof of Proposition \ref{pr:sta_me}. Since we give a proof of Proposition \ref{pr:sta_me} in the appendix, we leave the proof of Proposition \ref{pr:sta_me} to the readers.

\begin{prop} [Uniform-in-$\eps$ boundedness]
	\label{pr:Dodson_eps}
	Let $w_\eps \in \xX(\iI)$ be the solution to \eqref{eq:nls_w_eps}. Then for every $M>0$, there exists $\tilde{D}_M$ such that $\|w_\eps\|_{\xX_2(\iI)} \leq \tilde{D}_M$ whenever $\|w_\eps(a)\|_{L_x^2} \leq M$. 
\end{prop}

Proposition \ref{pr:Dodson_eps} is main technical part of this article. Combining the above two propositions, we have the following. 

\begin{prop} [Strong uniform-in-$\eps$ stability]
	\label{pr:strong_sta_eps}
	Let $w_\eps \in \xX(\iI)$ be the solution to \eqref{eq:nls_w_eps}, and $v_\eps \in \xX(\iI)$, $e_\eps \in L_{t}^{1}L_{x}^{2}(\iI)$ satisfying \eqref{eq:nls_v_eps}. Then for every $M>0$, there exist $\delta_m, C_M > 0$ such that if $\|w_\eps(a)\|_{L_{x}^{2}} \leq M$ and
	\begin{equation*}
	\|v_\eps(a) - w_\eps(a)\|_{L_{x}^{2}} + \|e\|_{L_{t}^{1}L_{x}^{2}(\iI)} \leq \delta_M, 
	\end{equation*}
	then we have
	\begin{equation*}
	\|v_\eps - w_\eps\|_{\xX(\iI)} \leq C_M \big( \|v_\eps(a) - w_\eps(a)\|_{L_{x}^{2}} + \|e\|_{L_{t}^{1}L_{x}^{2}(\iI)} \big). 
	\end{equation*}
	The constants $\delta_M$ and $C_M$ depend on $M$ only. 
\end{prop}

Again, the derivation of Proposition \ref{pr:strong_sta_me} from Proposition \ref{pr:sta_eps} and \ref{pr:Dodson_eps} is similar to the derivation of  Proposition~\ref{pr:strong_sta_me} from Proposition~\ref{pr:bd_me}, Propostion ~\ref{pr:sta_me}, which is put in the appendix.

\subsection{Propositions~\ref{pr:Dodson_eps}+\ref{pr:strong_sta_eps} imply Proposition~\ref{pr:bd_me}}

We now show that Proposition~\ref{pr:bd_me} is a consequence of Propositions~\ref{pr:Dodson_eps} and \ref{pr:strong_sta_eps}. Let $w$ be the solution to \eqref{eq:nls_w_me} with $\|w(a)\|_{L_x^2} \leq M$. We need to show $\|w\|_{\xX_2(\iI)} \leq \tilde{D}$ for some $\tilde{D}$ depending on $M$ only. 

If $m$ is smaller than some (possibly large) constant depending on $M$, then $w$ will satisfy a free equation after $\|w\|_{\xX_2}$ reaches that constant, and hence one has the desired control. Thus, it suffices to consider the case for large $m$. We re-write the equation for $w$ as
\begin{equation*}
i \d_t w + \Delta w = \theta_{m}(A) \nN^{\eps}(w) + e\;, \quad \|w(a)\|_{L_x^2} \leq M\;, 
\end{equation*}
where
\begin{equation*}
e = \Big( \theta_{m}\big( A + \|w\|_{\xX_2(a,t)}^{2} \big) - \theta_{m}(A) \Big) \nN^{\eps}(w)
\end{equation*}
satisfies the pointwise bound
\begin{equation*}
|e(t)| \leq \frac{C}{m} \|w\|_{\xX_2(a,t)}^{5} |w(t)|^{5-\eps}. 
\end{equation*}
We are now in the form of Proposition \ref{pr:strong_sta_eps} with $v_\eps = w$ and $\lambda_\eps = \theta_{m}(A)$. We need to show that $\|e\|_{L_{t}^{1}L_{x}^{2}(\iI)}$ is small for large $m$. Indeed, there exists $C_0>0$ universal such that
\begin{equation} \label{eq:bootstrap_e}
\|e\|_{L_{t}^{1}L_{x}^{2}(a,r)} \leq \frac{C_0}{m} \|w\|_{\xX(a,r)}^{10-\eps}
\end{equation}
for every $r \in [a,b]$. Let $w_\eps$ be the solution to \eqref{eq:nls_w_eps} with $w_\eps(a) = w(a)$. Let $B_M$, $\delta_M$ and $C_M$ be as in Propositions~\ref{pr:Dodson_eps} and \ref{pr:strong_sta_eps}. Let $m$ be sufficiently large so that
\begin{equation} \label{eq:choice_m}
\frac{C_0}{m} \cdot \Big( 2(M+ D_M + C_M \delta_M) \Big)^{10} \leq \delta_M. 
\end{equation}
This choice of $m$ depends on $M$ only. We now use a standard bootstrap argument to show that $\|w\|_{\xX(a,r)}$ will never exceed $2(M+D_M + C_M \delta_M)$ on $[a,b]$. We first note that $\|w\|_{\xX(a,a)} \leq M$. Suppose $r \in [a,b]$ is such that
\begin{equation*}
\|w\|_{\xX(a,r)} \leq 2 (M + D_M + C_M \delta_M). 
\end{equation*}
Then the bound \eqref{eq:bootstrap_e} and the choice of $m$ in \eqref{eq:choice_m} imply that
\begin{equation*}
\|e\|_{L_{t}^{1} L_{x}^{2}(a,r)} \leq \delta_M. 
\end{equation*}
Hence, by Proposition~\ref{pr:Dodson_eps}, we have
\begin{equation*}
\|w - w_\eps\|_{\xX(a,r)} \leq C_M \delta_M, 
\end{equation*}
which in turn implies
\begin{equation*}
\|w\|_{\xX(a,r)} \leq \|w_\eps\|_{\xX(a,r)} + C_M \delta_M \leq M + D_M + C_M \delta_M. 
\end{equation*}
Thus, we deduce that $\|w\|_{\xX(a,r)}$ can never reach $2 (M + D_M + C_M \delta_M)$ for $r \in [a,b]$, and in particular, we have the bound
\begin{equation*}
\|w\|_{\xX(\iI)} \leq 2(M + D_M + C_M \delta_M). 
\end{equation*}
This is true whenever $m$ satisfies \eqref{eq:choice_m}. On the other hand, if $m$ violates \eqref{eq:choice_m}, then $w$ satisfies the free equation after $\|w\|_{\xX_2(a,r)}$ reaches a large constant depending on $M$ only, and hence we also have the desired bounds in that case.

\section{Overview for the proof of Proposition \ref{pr:Dodson_eps}}\label{secroadmap2}

Fix any $\eps>0$, Proposition \ref{pr:Dodson_eps} clearly holds by the local theory of mass subcrtical NLS. Indeed, for any fixed $\eps_{0}>0$, the local theory of mass subcritical NLS implies \ref{pr:Dodson_eps} for $1>\eps>\eps_{0}$. On the other hand, when $\eps=0$, the end point case is just Theorem \ref{th:dodson}. Thus, the idea is to use some (concentration) compactness argument to push the potential "counter-example" of Propostion \ref{pr:Dodson_eps} to the end point $\eps=0$. We will indeed prove by contradiction and show no such counter example exists.

For notational simplicity, from now on, we assume without loss of generality that $\iI=[0,1]$. 

\subsection{Concentration compactness}
Now we introduce the preliminary  for concentration compactness. Concentration compactness is also refered as profile decompostion in the literature. We first introduce some notations.
\begin{defn}
We define unitary group acting on $L_{x}^{2}(\RRR^{d})$ as  $G:=\{g_{x_{0},\xi_{0},\lambda_{0}, t_{0}}| x_{0},\xi_{0}\in \RRR^{d}, \lambda_{0}\in \RR^{+}, t_{0}\in \RRR\}$, and 
\begin{equation}
(g_{x_{0},\xi_{0},\lambda_{0},t_{0}}f)(x) := \lambda_0^{-\frac{d}{2}} e^{ix\xi_{0}} \big(e^{i(-\frac{t_{0}}{\lambda_{0}^{2}})\Delta}f \big)\big(\frac{x-x_{0}}{\lambda_{0}}\big).
\end{equation}
\end{defn}

Now we are ready to state the concentration compactness. 

\begin{prop}[\cite{merle1998compactness},\cite{carles2007role}, \cite{begout2007mass}, \cite{tao2008minimal}]\label{Prop: profiledecomposition}
Let $(q,r)$ be an admissible pair in dimension $d$ in the sense that
\begin{equation*}
\frac{2}{q} + \frac{d}{r} = \frac{d}{2}. 
\end{equation*}
Let $\{f_{n}(x)\}_{n=1}^{\infty}$ be bounded in $L_{x}^{2}(\RRR^{d})$. Up to picking subsequence, there exist a family of $L^{2}(\RRR^{d})$ functions $\{\phi_{j}\}_{j}$, 
 and group elements $g_{j,n}=g_{x_{j,n},\xi_{j,n},\lambda_{j,n},t_{j,n}}$, such that for all $ J=1,2,\cdots$, the decomposition
\begin{equation}\label{eq: tempdeco}
f_{n}= \sum_{j=1}^{J}g_{j,n}\phi_{j}+\omega_{n}^{J},
\end{equation}
satisfies the following properties: 
\begin{equation}\label{eq: orthpara}
\forall j\neq j',
\lim_{n\rightarrow \infty}\frac{\lambda_{j,n}}{\lambda_{j',n}}+\frac{\lambda_{j',n}}{\lambda_{j,n}}+|\frac{x_{j,n}-x_{j',n}}{\lambda_{j,n}}|+|\lambda_{j,n}(\xi_{j,n}-\xi_{j',n})|+|\frac{t_{j,n}-t_{j',n}}{\lambda_{j,n}^{2}}|=\infty.
\end{equation}
\begin{equation}\label{eq: orthstri}
\forall j\neq j', \lim_{n\rightarrow \infty}\|e^{it\Delta}(g_{j,n}\phi_{j})e^{it\Delta}(g_{j',n}\phi_{j'})\|_{L_{t}^{\frac{q}{2}}L_{x}^{\frac{r}{2}}(\RR \times \RR^d)}=0.
\end{equation}
\begin{equation}\label{eq: orthmass}
\forall J\geq 1, \lim_{n\rightarrow \infty}|\|f_{n}\|_{L_{x}^{2}}^{2}-\sum_{j\leq J}\|\phi_{j}\|_{L_{x}^{2}}-\|\omega^{J}_{n}\|_{L_{x}^{2}}^{2}|=0.
\end{equation}

\begin{equation}\label{eq:smallnessofmass}
\lim_{J \rightarrow\infty} \limsupn \|e^{it\Delta}\omega_{n}^{J}\|_{L_{t}^{q}L_{x}^{r}(\RR \times \RR^d)}=0.
\end{equation}

\begin{equation}\label{eq: weakconve}
\forall j\leq J, g_{j,n}^{-1}\omega^{J}_{n}\rightharpoonup 0 .
\end{equation}
We call  $\{f_{n}\}_{n}$ admits a profile decomposition with profiles $\{\phi_{j}; \{x_{j,n}, \xi_{j,n}, \lambda_{j,n},t_{j,n}\}_{n}\}_{j}$.
\end{prop}

\begin{rem}\label{symmetryrmk}
Those parameters takes into account the symmetry for both mass critical nonlinear Schr\"odinger equation and linear Schr\"odinger equation. In particular, let $\phi\in L_{x}^{2}$, and let $x_{0},\xi_{0}, \lambda_{0}, t_{0}$ be parameters. Let $\Phi$ solves for 
\begin{equation}
\begin{cases}
i\partial_{t}\Phi+\Delta \Phi=c|\Phi|^{4/d}\Phi,\\
\Phi(-\frac{t_{0}}{\lambda_{0}^{2}})=\phi(x)
\end{cases}
\end{equation}

Then we have
$\Psi(t,x)=\frac{1}{\lambda_{0}^{d/2}}e^{i\xi_{0}x}e^{-i\xi_{0}^{2}t}\Phi(\frac{t-t_{0}}{\lambda_{0}^{2}},\frac{x-x_{0}-2\xi_{0}t}{\lambda_{0}})$ solves

\begin{equation}
\begin{cases}
i\partial_{t}\Psi+\Delta \Psi=-|\Psi|^{4/d}\Psi,\\
\Psi(0)=\frac{1}{\lambda_{0}^{d/2}}\phi(\frac{x-x_{0}}{\lambda_{0}}),
\end{cases}
\end{equation}

\end{rem}

\begin{rem}\label{rem: symmetrykeepsnorm}
For any $(q,r)$ are admissable (that is, $\frac{2}{q}+\frac{d}{r}=\frac{d}{2}$), and $\Phi(t,x)\in L_{t}^{q}L_{x}^{r}$, via a change of variables, one has
\begin{equation} \label{eq: naturalscaling}
\|\Phi\|_{L_{t}^{q}L_{x}^{r}(\RRR\times \RRR^{d})}=  \Big\|\lambda_{0}^{-\frac{d}{2}}e^{i\xi_{0}x}e^{-i\xi_{0}^{2}t} \cdot \Phi\Big(\frac{t-t_{0}}{\lambda_{0}^{2}},\frac{x-x_{0}-2\xi_{0}t}{\lambda_{0}}\Big) \Big\|_{L_{t}^{q}L_{x}^{r}(\RR \times \RR^d)}.
\end{equation}
\end{rem}
\begin{rem}\label{rem: orthogonality}
Estimate \ref{eq: orthstri} does not indicate $\phi_{j}$ and $\phi_{j}'$ are orthogonal in any sense. It is totally possible that $\phi_{j}=\phi_{j'} $ for some $j\neq j'$. Estimate \eqref{eq: orthstri} is a direct consequence of \eqref{eq: orthpara}. Indeed, for $(q,r)$ admissible and $\Psi, \Phi$ in $L_{t}^{q}L_{x}^{r}$. Let $\Psi_{j,n}=\frac{1}{\lambda_{j,n}^{d/2}}e^{i\xi_{j,n}x}e^{-i\xi_{j,n}^{2}t}\Psi(\frac{t-t_{jn}}{\lambda_{j,n}^{2}},\frac{x-x_{j,n}-2\xi_{j,n}t}{\lambda_{j,n}})$, and $\Phi_{j,n}=\frac{1}{\lambda_{j,n}^{d/2}}e^{i\xi_{j,n}x}e^{-i\xi_{j,n}^{2}t}\Phi(\frac{t-t_{j, n}}{\lambda_{j,n}^{2}},\frac{x-x_{j,n}-2\xi_{j,n}t}{\lambda_{j,n}})$. Then \eqref{eq: orthpara} implies
\begin{equation}
\lim_{n\rightarrow \infty} \|\Phi_{j,n} \Psi_{j',n}\|_{L_{t}^{q/2}L_{x}^{r/2}(\RRR\times \RRR^{d})}=0, \forall j\neq j'.
\end{equation}
\end{rem}

\begin{rem}\label{rem: useful}
A typical feature in the application of concentration compactness is that a lot of subsequence will be taken. So it is typical to assume all the limit involved, actually exists or equal to $\pm \infty$. For example,  we may always assume $\limn\frac{-t_{j,n}}{\lambda_{j,n}^{2}}$ exists or equal to $\pm \infty$. Up to further adjusting profile, we may assume  $t_{j,n}\equiv 0$ if $lim_{n}\frac{-t_{j,n}}{\lambda_{j,n}^{2}}$ exists.
\end{rem}

With Remark \ref{rem: useful}, one is led to  the following standard notion:
\begin{defn}
We call a profile $f$ with parameter $\{x_{n},\xi_{n},\lambda_{n},t_{n}\}_{n}$
\begin{enumerate}
\item Compact profile if $t_{n}\equiv 0$,
\item Backward scattering profile if $\lim_{n\rightarrow \infty}-\frac{t_{n}}{\lambda_{n}^{2}}=-\infty$.
\item Forward scattering profile if $\lim_{n\rightarrow \infty}-\frac{t_{n}}{\lambda_{n}^{2}}=\infty$,
\end{enumerate}
\end{defn}

\subsection{A stability type proposition}
Let $\{f_{n}\}_{n}$ be bounded in $L_{x}^{2}$  and  admits profile decomposition with profiles $\{\phi_{j}\}_{j}$  with parameters 
$\{ x_{j,n},\xi_{j,n},\lambda_{j,n}, t_{j,n} \}_{j,n}$. Let $\eps_{n}\rightarrow 0$. We needs to understand the solution $u_{n}$ to the following equations within time interval $[0,1]$, for $n$ large
\begin{equation}\label{eq: epequation}
\begin{cases}
i\partial_{t}u_{n}+\Delta u_{n}=N^{\eps_{n}}(u_{n}),
u_{n}(0,x)=f_{n}
\end{cases}
\end{equation}
The main difference, compare the with the usual setting of profile decomposition applied to mass critical nonlinear Schr\"odinger equations (some reference should be put here), is the parameter $\eps_{n}$ destroys the scaling symmetry of the equation. Recall if $u(t,x)$ solves
\begin{equation}
i\partial_{t}u+\Delta u= |u|^{4}u
\end{equation}
with initial data $u(0)=u_{0}$, then $\frac{1}{\lambda^{1/2}}u(\frac{t}{\lambda^{2}}, \frac{x}{\lambda})$ solves the same equation with initial data $\frac{1}{\lambda}u_{0}(x/\lambda).$ This does not work for  \eqref{eq: epequation}.

Still, before we state the main proposition in this subsection, we introduce the usual notion of nonlinear profile. But we are only interested in backward scattering profile and compact profile for the purpose of our work. 

\begin{defn}
Let $0<c^{*}\leq 1$ be a number. Suppose we are given a  linear profile, i.e. an $L^{2}$ function $\phi(x)$ with parameters $\{x_{n},\xi_{n}, \lambda_{n}, t_{n}\}_{n}$. If it is a compact profile, i.e. $t_{n}\equiv 0$, we define its associated nonlinear profile with parameter $c^{*}$, as the solution $\Phi$ which solves
\begin{equation}
\begin{cases}
i\partial_{t}\Phi(t,x)+\Delta \Phi(t,x)=c^{*}|\Phi|^{4}\Phi, \\
\Phi(0) = \phi, 
\end{cases}
\end{equation}
If it is a back forward scattering profile, we define its associated  nonlinear profile with parameter $c^{*}$ as the solution $\Phi$ to
\begin{equation}
i\partial_{t}\Phi(t,x)+\Delta \Phi(t,x)=c^{*}|\Phi|^{4}\Phi\\
\end{equation}
and 
\begin{equation}
\lim_{n\rightarrow \infty}\|\Phi(-\frac{t_{n}}{\lambda_{n}^{2}})-e^{(-t_{n}/\lambda_{n}^{2})i\Delta}\phi(x)\|_{L_{x}^{2}}=0.
\end{equation}
\end{defn}
\begin{rem} 
The parameter $x_{n}$ and $\xi_{n}$ will not have a effect in this definition.
\end{rem}
\begin{rem}
The existence and uniqueness of such nonlinear profile is well known and essentially  equivalent to local theory for mass critical nonlinear Schrodinger equation, one may  refer to  Notation 2.6 in \cite{duyckaerts2009universality}. Usually, one only consider the case for $c^{*}=1$, but we have introduce this extra parameter for later use. 
\end{rem} 

Now, we are ready to state our results
\begin{prop}\label{propositon: techicnalmain}
Let $\{f_{n}\}_{n}$ be a bounded sequence in $L_{x}^{2}(\RRR)$ admits profile decomposition with profiles $\{\phi_{j}; \{x_{j,n}, \xi_{j,n}, \lambda_{j,n},t_{j,n}\}_{n}\}_{j}$. Note that for all $J>0$, we have 
\begin{equation}
f_{n}(x)=\sum_{j\leq J}\frac{1}{\lambda_{j,n}^{\frac{1}{2}}}e^{ix\xi_{j,n}}(e^{i(-\frac{t_{j,n}}{\lambda_{j,n}^{2}})\Delta}\phi_{j})(\frac{x-x_{j,n}}{\lambda_{j,n}})+\omega^{J}_{n}.
\end{equation}
Let $\eps_{n}\rightarrow 0$. Let $u_{n}$ be the solution to
\begin{equation}
\begin{cases}
i\partial_{t}u_{n}+\Delta u_{n}=|u_{n}|^{4-\eps_{n}}u_{n},\\
u_{n}(0)=f_{n}
\end{cases}
\end{equation}

Then one has that for every $\kappa>0$, there exists some 	$J>0$, so that
\begin{equation}\label{eq: ultrastable}
\limsup_{n\rightarrow \infty}\|u_{n}-\sum_{j\leq J}\Psi_{j,n}-e^{it\Delta}\omega_{n}^{J}\|_{L_{t}^{5}L_{x}^{10}[0,1]}\leq \kappa
\end{equation}

Here the $\Psi_{j,n}$ is defined as following:
For notation convenience, for each profile $\phi_{j}$ with parameters $ \{x_{j,n}, \xi_{j,n}, \lambda_{j,n},t_{j,n}\}_{n}$. We denote $\phi_{j,n}:=\frac{1}{\lambda_{j,n}^{\frac{1}{2}}}e^{ix\xi_{j,n}}(e^{i(-\frac{t_{j,n}}{\lambda_{j,n}^{2}})\Delta}\phi_{j})(\frac{x-x_{j,n}}{\lambda_{j,n}})$, we denote $\Phi_{j}^{L}:=e^{it\Delta}\phi_{j}$
\begin{enumerate}
\item If $\lambda_{j,n}\xrightarrow{n\rightarrow \infty} \infty $, we let 
\begin{equation}
\Psi_{j}=\Phi_{j}^{L}, \Psi_{j,n}=e^{it\Delta}\phi_{j,n}\equiv \frac{1}{\lambda_{j,n}^{1/2}}e^{ix {\xi_{j,n}}}e^{-it|\xi_{j,n}|^{2}}\Psi_{j}(\frac{t-t_{j,n}}{\lambda_{j,n}^{2}},\frac{x-x_{j,n}-2\xi_{j,n}t}{\lambda_{j,n}})
\end{equation}

\item If $\lambda_{j,n}=1, \forall n$, we let  $\Phi_{j} $ be the associated nonlinear profile with parameter $1$, and let
\begin{equation}
\Psi_{j}=\Phi_{j}, \Psi_{j,n}=e^{i \xi_{j,n} x} e^{- i \xi_{j,n}^{2} t}\frac{1}{\lambda_{j,n}^{1/2}}\Psi_{j}(\frac{t-t_{j,n}}{\lambda_{j,n}^{2}},\frac{x-x_{j,n}-2\xi_{j,n}t}{\lambda_{j,n}})
\end{equation}

\item If $\lambda_{j,n}\rightarrow 0$, we distinguish two different cases, (up to picking subsequence, we can always assume at least
 one of the following cases exists)
\begin{itemize}
\item $\eps_{n}$ asymptotically large in the sense $\lim_{n\rightarrow} \lambda_{n}^{\eps_{n}}=0.$ we let
\begin{equation}
\Psi_{j}=\Phi_{j}^{L}, \Psi_{j,n}= \frac{1}{\lambda_{j,n}^{1/2}}e^{ix {\xi_{j,n}}}e^{-it|\xi_{j,n}|^{2}}\Psi_{j}(\frac{t-t_{j,n}}{\lambda_{j,n}^{2}},\frac{x-x_{j,n}-2\xi_{j,n}t}{\lambda_{j,n}})
\end{equation}
\item $\eps_{n}$ asymptotically small in the sense $\lim_{n\rightarrow} \lambda_{n}^{\eps_{n}}=c^{2}_{j}>0.$, we let $\Phi_{j}$ be the associated nonlinear profile with parameter $c_{j}>0$, let 
\begin{equation}
\Psi_{j}=\Phi_{j},
\Psi_{j,n}= \frac{1}{\lambda_{j,n}^{1/2}}e^{ix {\xi_{j,n}}}e^{-it|\xi_{j,n}|^{2}}\Psi_{j}(\frac{t-t_{j,n}}{\lambda_{j,n}^{2}},\frac{x-x_{j,n}-2\xi_{j,n}t}{\lambda_{j,n}})
\end{equation}
\end{itemize}
\end{enumerate}
Finally,  with \eqref{eq: ultrastable}, one can derive
\begin{equation}\label{eq: finalnormcontrol}
\limsup_{n} \|u_{n}\|^{5}_{L_{t}^{5}L_{x}^{10}}\lesssim\limsup_{n}\sum_{j=1}^{\infty}\|\Psi_{j,n}\|^{5}_{L_{t}^{5}L_{x}^{10}(\spacetime)}\leq \sum_{j=1}^{\infty}\|\Psi_{j,n}\|^{5}_{L_{t}^{5}L_{x}^{10}(\RRR\times \RRR)}<\infty.
\end{equation}
\end{prop}

\begin{rem}
One could always picking subsequence so that $\lim_{n}\lambda_{j,n}^{\eps_{n}}$ exists (or equal to $\infty$) for every $j$.
\end{rem}

\begin{rem}
In \eqref{eq: finalnormcontrol},   note that for each $j$, $\|\Phi_{j,n}\|_{L_{t}^{5}L_{x}^{10}(\RRR\times \RRR)}$ is independent of $n$.
\end{rem}
\begin{rem}
If one put $\eps_{n}\equiv 0$, and $\Psi_{j,n}=\frac{1}{\lambda_{j,n}^{1/2}}e^{ix {\xi_{j,n}}}e^{-it|\xi_{j,n}|^{2}}\Phi_{j}(\frac{t-t_{j,n}}{\lambda_{j,n}^{2}},\frac{x-x_{j,n}-2\xi_{j,n}t}{\lambda_{j,n}})$, then such a result  is known in the literature, after some natural modification.
We also remark, the way we present Proposition \ref{propositon: techicnalmain} is not totally perturbative, since we will use Theorem \ref{th:dodson} in the proof. 
\end{rem}

\subsection{Proposition~\ref{propositon: techicnalmain} implies Proposition~\ref{pr:Dodson_eps}} \label{secpropmain}

We prove by contradiction. Suppose Proposition~\ref{pr:Dodson_eps} is not true, then there exists $M>0$ and a sequence $\{\eps_n\} \in (0,1)$ with
\begin{equation}
i \d_{t} u_{n}+\Delta u_{n} = N^{\eps_{n}}(u_{n})\;, \quad u_n(a) = w(a)
\end{equation}
such that
\begin{equation}
\begin{cases}
\|u_{n}\|_{\xX_2(\iI)} \geq n,\\
\|u_{n}(0)\|_{L_{x}^{2}}\leq M+\frac{1}{n} \leq 2(M+1).
\end{cases}
\end{equation}
Up to a subsequence, one may assume $\eps_{n}\rightarrow 0$. Then this contradicts \eqref{eq: finalnormcontrol} in Proposition~\ref{propositon: techicnalmain}.

\section{Proof of Proposition~\ref{propositon: techicnalmain}} \label{sectechmain}

We put $\iI=[0,1]$. Propostiton \ref{propositon: techicnalmain} is  the consequence of the key orthogonality condition \eqref{eq: orthpara} in profile decomposition, the uniform stability Prop \ref{pr:sta_eps}, and the following key Lemma.
\begin{lem}\label{lem: oneprofilekey}
Let $u_{n}, f_{n},\phi_{j}, \phi_{j,n} ,\Psi_{j,n}, \eps_{n}$ be as in Proposition \ref{propositon: techicnalmain}.  Let $v_{j,n}$ solves
\begin{equation}\label{eq: vjn}
\begin{cases}
i\partial_{t}v_{j,n}+\Delta v_{j,n}=N^{\eps_{n}}(v_{j,n}),\\
v_{j,n}(0)=\phi_{j,n}.
\end{cases}
\end{equation}
Then we have for each $j$,
\begin{equation}\label{eq: asyappoxi}
\lim_{n\rightarrow \infty}\|\Psi_{j,n}-v_{j,n}\|_{\xX(\iI)}=0.
\end{equation}
\end{lem}

\subsection{Proof of Proposition \ref{propositon: techicnalmain} assuming Lemma~\ref{lem: oneprofilekey}}
Before we start, let $u_{n}, f_{n},\phi_{j}, \phi_{j,n} ,\Psi_{j,n}, \eps_{n}$ be as in Proposition \ref{propositon: techicnalmain}.
Also, let $v_{j,n}$ be defined as in \eqref{eq: vjn}.
We assume that  $\|f_{n}\|_{L_{x}^{2}}\leq M_{0}$. Note that \eqref{eq: orthmass} implies for all $j$, $\|\phi_{j}\|_{L_{x}^{2}}\leq M_{0}$.
which gives
\begin{equation}\label{eq: initialmass}
\|v_{j,n}(0)\|_{L_{x}^{2}(\RRR)}=\|\Psi_{j}\|_{L_{x}^{2}(\RRR)}=\|\phi_{j}\|_{L_{x}^{2}}
\end{equation}
In the spirit of Proposition \ref{pr:sta_eps}, we want to use , for$J$ large but fixed (later)
\begin{equation}
\sum_{j\leq J}\Psi_{j,n}+e^{it\Delta}\omega_{n}^{J}
\end{equation}
as an approximate solution to $u_{n}$. Such strategy is typical in the application of concentration compactness.

At the technical level, it is more convenient to use
\begin{equation}
w_{J,n}:=\sum_{j\leq J}v_{j,n}+e^{it\Delta}\omega_{n}^{J}
\end{equation}
as an approximate solution to $u_{n}$. The goal is to prove
\begin{equation}\label{eq: goal}
\lim_{J\rightarrow \infty}\lim_{n\rightarrow \infty} \|w_{J,n}-u_{n}\|_{\xX(\iI)}=0
\end{equation}
Note that \eqref{eq: ultrastable} follows from \eqref{eq: goal} since we have \eqref{eq: asyappoxi}. We emphasize here in \eqref{eq: goal} one first pushes $n$ to infinity, then pushes $J$.

Lemma \ref{lem: oneprofilekey} roughly says $v_{j,n}$ and $\Psi_{j,n}$ are asymptotically equivalent. Before we proceed, we give a few remarks here. 
\begin{itemize}
\item $\Psi_{j,n}$ is easier to study analytically, since it has the structure
\begin{equation}\label{eq: structure}
\Psi_{j,n}= \frac{1}{\lambda_{j,n}^{1/2}}e^{ix {\xi_{j,n}}}e^{-it|\xi_{j,n}|^{2}}\Psi_{j}(\frac{t-t_{j,n}}{\lambda_{j,n}^{2}},\frac{x-x_{j,n}-2\xi_{j,n}t}{\lambda_{j,n}})
\end{equation}
\item $v_{j,n}$ are more suitable to do approximation. Because,  for each $j$, $\Psi_{j,n}$ either solves linear Schr\"odinger and mass critical nonlinear Schr\"odinger, and $v_{j,n}$ solves
\begin{equation}
i\partial_{t}v_{j,n}+\Delta v_{j,n}=N^{\eps_{n}}v_{j,n}
\end{equation}
\end{itemize}
We also point out that $v_{j,n}$ is only asymptotically  close to $\Psi_{j,n}$ within time interval $\iI$.

We write down the equation of $w_{J,n}$.
\begin{equation}
\begin{cases}
i\partial_{t}w_{J,n}+\Delta w_{J,n}=N^{\eps_{n}}(w_{J,n})+e_{J,n},\\
w_{J,n}(0)=u_{n}(0)=f_{n}.
\end{cases}
\end{equation}
and $e_{J,n}$ is of form
\begin{equation}\label{eq: errorformula}
\begin{aligned}
-e_{J,n}=
&|\sum_{j\leq J} v_{j,n}+e^{it\Delta}\omega_{J,n}|^{4-\eps_{n}}(\sum_{j\leq J} v_{j,n}+e^{it\Delta}\omega_{J,n})-\sum_{j\leq J}|v_{j,n}|^{4-\eps_{n}}v_{j,n}
\\
=&\left(|\sum_{j\leq J} v_{j,n}+e^{it\Delta}\omega_{J,n}|^{4-\eps}(\sum_{j\leq J} v_{j,n}+e^{it\Delta}\omega_{J,n})-|\sum_{j\leq J} v_{j,n}|^{4-\eps_{n}}(\sum_{j\leq J} v_{j,n})\right)\\
+&\left(|\sum_{j\leq J} v_{J,n}|^{4-\eps_{n}}(\sum_{j\leq J} v_{j,n})-\sum_{j\leq J}|v_{j,n}|^{4-\eps_{n}}v_{J,n}\right)
\\
=&(I)+(II)
\end{aligned}
\end{equation}
To apply Proposition \ref{pr:sta_eps}, we need a good estimate for $\|w_{J,n}\|_{\xX_2(\iI)}$ and $\|e_{J,n}\|_{L_{t}^{1}L_{x}^{2}(\iI)}$. To achieve this ,we observe
\begin{itemize}
\item There exists $C_{M_{0},1}$ For each $j$, 
\begin{equation}\label{eq: submass}
\lim_{n\rightarrow \infty}\|v_{j,n}\|_{\xX_2(\iI)}\leq C_{M_{0},1}
\end{equation}
\item There exists some $\delta_{0}$, so that if $\|v_{j,n}(0)\|_{L_{x}^{2}}\leq \delta_{0}$, then
\begin{equation}\label{eq: estiamte2}
\|v_{j,n}\|_{\xX_2(\iI)} \lesssim \|v_{j,n}(0)\|_{L_{x}^{2}}.
\end{equation}
\item For all $j\neq j'$,
\begin{equation}\label{eq: estimate3}
\lim_{n\rightarrow \infty}\|v_{j,n} v_{j',n}\|_{L_{t}^{5/2}L_{x}^{5}(\iI)}=0
\end{equation}
\item There exists some $C_{M_{0},2}$, so that for all $J$ large and fixed, one has 
\begin{equation}\label{eq: rightguess}
\limsup_{n\rightarrow \infty} \big\|\sum_{j\leq J}v_{j,n} \big\|_{\xX_2(\iI)} \leq C_{M_{0},2}
\end{equation}
\item There exists some $C_{M_{0},0}$, so that for all $J$ large and fixed, one has
\begin{equation}\label{eq: rightguess2}
\limsup_{n}\|w_{J,n}\|_{\xX_2(\iI)}\leq C_{M_{0},0}.
\end{equation}
\end{itemize}

We now prove the above estimates one by one. The key tool is Lemma \ref{lem: oneprofilekey}.

First, by \eqref{eq: initialmass}, recalling that $\Psi_{j}$ either solves linear Schroindger equation or mass critical nonlinear Schrodinger equation, Theorem \ref{th:dodson} and Strichartz esitmates implies
\begin{equation}
\|\Psi_{j}\|_{\xX_2(\RR)} \leq C_{M_{0},1}, \text{for some} \phantom{1} C_{M_{0},1}, 
\end{equation}
which implies
\begin{equation}
\|\Psi_{j,n}\|_{\xX_2(\RR)}=\|\Psi_{j}\|_{\xX_2(\RR)}\leq C_{M_{0},1}
\end{equation}
for all $n$. Thus, we derive \eqref{eq: submass} by applying \eqref{eq: asyappoxi}. 

For the second one, \eqref{eq: estiamte2} follows directly from Lemma \ref{lem: enhancedsmallness}.

For the third one, by \eqref{eq: asyappoxi}, estimate \eqref{eq: estimate3} follows from the fact
\begin{equation}
\lim_{n\rightarrow \infty }\|	\Psi_{j,n} \Psi_{j',n}\|_{L_{t}^{5/2}L_{x}^{5}(\iI)}\leq \lim_{n\rightarrow \infty}\|\Psi_{j,n} \Psi_{j',n}\|_{L_{t}^{5/2}L_{x}^{5}(\RR)}=0,
\end{equation}
which is a consequence of \eqref{eq: structure} and \eqref{eq: orthpara} (see also Remark \ref{rem: orthogonality}). 

For the fourth one, let us fix the $\delta_{0}$ small enough according to Lemma~\ref{lem: enhancedsmallness}. Note that due to \eqref{eq: orthmass}, there can only be finite many $j$s so that $\|\phi_{j}\|_{L_{x}^{2}}\geq \delta_{0}$. We assume that 
\begin{equation}\label{eq: remainmass}
\|\phi_{j}\|_{L_{x}^{2}}\leq \delta_{0}, j\geq J_{0}+1
\end{equation}
Now, for any $J\geq J_{0}+1$, combining \eqref{eq: submass} and \eqref{eq: remainmass}, we have 
\begin{equation}\label{eq: guess2}
\begin{aligned}
&\lim_{n\rightarrow \infty}\sum_{j=1}^{J}\|v_{j,n}\|_{\xX_2(\iI)}^{5}\\
&\lesssim  \sum_{j=1}^{J_{0}}C_{M_{0},1}^{5}+\sum_{j\geq J_{0}+1} \|\phi_{j}\|_{L_{x}^{2}}^{5}
&\lesssim   \sum_{j=1}^{J_{0}}C_{M_{0},1}^{5}+M_{0}^{2}. 
\end{aligned}
\end{equation}
Note that the right side of \eqref{eq: guess2} does not depend on $J$. Thus, \eqref{eq: rightguess} follows from \eqref{eq: guess2} and \eqref{eq: estimate3}.

Finally, estimate \eqref{eq: rightguess2} follows from the triangle inequality, \eqref{eq:smallnessofmass} and \eqref{eq: rightguess}.

With \eqref{eq: rightguess2}, in order to apply Proposition~\ref{pr:sta_eps} to prove \eqref{eq: goal}, we only need to show
\begin{equation}\label{eq: controloferrorterm}
\lim_{J}\lim_{n}\|e_{J,n}\|_{L_{t}^{1}L_{x}^{2}(\iI)}=0.
\end{equation}
Recall the formula of $e_{J,n}$, \eqref{eq: errorformula},
Fix any $J$, term $\|(II)\|_{L_{t}^{1}L_{x}^{2}(\iI)}$ goes to zero as $n\rightarrow \infty$ by \eqref{eq: estimate3} and \eqref{eq: submass}.
On the other hand, due to \eqref{eq: rightguess} and \eqref{eq:smallnessofmass}, we have
\begin{equation}
\limsup_{J}\limsup_{n}\|(I)\|_{L_{t}^{1}L_{x}^{2}(\iI)} \rightarrow 0.
\end{equation}
Thus, \eqref{eq: controloferrorterm} follows. We are left with the proof of \eqref{eq: finalnormcontrol}, but the proof is essentially same as the proof \eqref{eq: rightguess2}; see also \eqref{eq: guess2}.

\subsection{Proof of Lemma~\ref{lem: oneprofilekey}}

Fix $j=1$. Since we are now working one profile, we may assume $x_{1,n}\equiv\xi_{1,n}\equiv 0.$

The reason we could, without loss of generality, assume  $x_{j,n}, \xi_{j,n}\equiv 0$ is because those two symmetry work on exactly same way for mass critical NLS, mass subscrtical NLS and linear Schr\"odinger, and leave our considered norm invariant. The following remark makes it precise. 

\begin{rem}
If $v$ solves
\begin{equation}
i\partial_{t}v+\Delta v=\mu |v|^{4-\eps}v, 
\end{equation}
for some $\eps\in [0,1]$, $\mu=0$ or $1$.
Then the solution to equation
\begin{equation}
i\partial_{t}w+\Delta w=\mu|w|^{4-\eps}w,
w(0)=e^{i\xi_{0}}v(0,x-x_{0})
\end{equation}
is of form $w(t,x)=e^{i\theta(t)}e^{i\xi(t)x}v(t,x-x(t))$ and the formula of $\theta(t), x(t),\xi(t)$ only depends on $x_{0},\xi_{0}$. 
\end{rem}

Note that $\Psi_{1,n}$ is defined via different way, depending on the asymptotic behavior of $\lambda_{1,n}$ and $\eps_{1,n}$. We will do a case by case study of Lemma \ref{lem: oneprofilekey}

\subsubsection{Case 1: $\lim_{n\rightarrow \infty}\lambda_{1,n}=\infty$}
In this case, we want to prove the dynamic is essentially linear.  Recall $\Phi_{1,L}=e^{it\Delta}\phi$
Recall $v_{1,n}$ solves in time interval $I$
\begin{equation}
i\partial_{t}v_{1,n}+\Delta v_{1,n}=N^{\eps_{n}}(v_{1,n}), \quad
v_{1,n}(0,x)=\frac{1}{\lambda_{1,n}^{1/2}}\Phi_{1,L}(-\frac{t_{1,n}}{\lambda_{1,n}},\frac{1}{\lambda_{1,n}}).
\end{equation}
The desired results follows from Lemma \ref{lem: prelwp} if we can show
\begin{equation}\label{eq: vanishoflinearscatter}
\|e^{it\Delta}[v_{1,n}(0)]\|_{\xX_2(\iI)}\xrightarrow{n\rightarrow \infty} 0.
\end{equation}
First note that via Strichartz estimate, we have
\begin{equation}
\|\Phi_{1,L}\|_{\xX_2(\RR)} < +\infty,
\end{equation}
Then we note that
\begin{equation}\label{eq: rescaling}
\|e^{it\Delta}[v_{1,n}(0)]\|_{\xX_2(\iI)}=\|\Phi_{1,L}\|_{\xX_2\big(\big[-\frac{t_{1,n}}{\lambda_{1,n}^{2}},-\frac{t_{1,n}-1}{\lambda_{1,n}^{2}}\big)\big)}. 
\end{equation}
If $\phi_{1,n}$is a forward scattering profile, i.e. $\lim_{n\rightarrow\infty}-\frac{t_{1,n}}{\lambda_{1,n}}=\infty$, \eqref{eq: vanishoflinearscatter} always holds and we don't even need any information about $\lambda_{1,n}$. In particular, we will not discuss forward scattering profiles in the later two cases.
 
If $\phi_{1,n}$ is a compact profile, i.e. $t_{1,n}\equiv 0$, we have
\begin{equation}
\|e^{it\Delta}v_{1,n}\|_{\xX_2(\iI)}=\|\Phi_{1,L}\|_{\xX_2\big(\big[0,\frac{1}{\lambda_{1,n}^{2}}\big)\big)} \rightarrow 0
\end{equation}
since $\frac{1}{\lambda_{1,n}^{2}}\rightarrow 0$ via $\lim_{n\rightarrow \infty}\lambda_{1,n}=\infty$. 

Finally, for a backward scattering profile, we still have $\lim_{n\rightarrow \infty}\frac{t_{1,n}-1}{\lambda_{1,n}^{2}}=-\infty$ since $\frac{1}{\lambda_{1,n}^{2}}\rightarrow 0$. This implies \eqref{eq: vanishoflinearscatter}.

We also take this chance to note that, whenever one has $\lim_{n\rightarrow \infty} \big(- \frac{t_{1,n}-1}{\lambda_{1,n}^{2}} \big)=-\infty$ for a backward scattering profile, the dynamic is linear since we have \eqref{eq: vanishoflinearscatter}. Thus, in the later two cases , for a backward scattering profile, we need only to consider the sub-situation
\begin{equation}
-\frac{t_{1,n}-1}{\lambda_{1,n}^{2}}>-\infty.
\end{equation}
Since $\lim_{n\rightarrow \infty}-\frac{t_{1,n}}{\lambda_{1,n}^{2}}=-\infty$, this would imply
\begin{equation}\label{eq: caseofinterest}
\lim_{n\rightarrow \infty}\lambda_{1,n}=0 \text{  and } \liminf_{n\rightarrow \infty}{t_{1,n}}\leq 1.
\end{equation}

\subsubsection{Case 2: $\lambda_{1,n}\equiv 1$}

As explained in the  previous case, we only need to handle compact profile in this case, i.e. $t_{1,n} \equiv 0$. The desired result follows from the lemma below. 

\begin{lem}\label{lem: epconvengence}
Let $\iI=[0,1]$, $v_{0}\in L_{x}^{2}$, and $v_\eps$ solves the equation
\begin{equation}
\begin{cases}
i\partial_{t}v_{\eps}+\Delta v_{\eps}=N^{\eps }(v_{\eps}),\\
v_{\eps}(0)=v_{0}
\end{cases}
\end{equation}
in $\iI$. Let $v$ solve
\begin{equation}
\begin{cases}
iv_{t}+\Delta v=|v|^{4}v,\\
 v(0)=v_{0}\in L_{x}^{2}. 
\end{cases}
\end{equation}
Then, 
\begin{equation}
\|v_\eps-v\|_{\xX(\iI)} \rightarrow 0
\end{equation}
as $\eps \rightarrow 0$. 
\end{lem}
\begin{proof}
Let $\|v_{0}\|_{L_{x}^{2}}=M_{0}$
First note that for any $\delta>0$, one will be able find some smooth $\tilde{v}$ so that
\begin{equation}\label{eq: eqfortv1222}
\begin{cases}
i\tilde{v}+\Delta \tilde{v}=|\tilde{v}|^{4}\tilde{v},\\
\tilde{v}(0)=\tilde{v}_{0}
\end{cases}
\end{equation}
and 
\begin{equation}\label{eq: approximate11222}
\begin{cases}
\|\tilde{v}_{0}-v_{0}\|_{L_{x}^{2}}\leq \delta,\\
\|\tilde{v}-v\|_{L_{t}^{5}L_{x}^{10}\cap L_{t}^{\infty}L_{x}^{2}(\RRR\times \RRR)}\leq \delta,\\
\|\tilde{v}\|_{L_{t}^{\infty}H^{100}_{x}}\leq C_{\delta}. 
\end{cases}
\end{equation}
A typical way to construct such $\tilde{v}$ is to choose $N_{\delta}$ depending on $\delta$ large enough, and let $\tilde{v}(0)=P_{<N_{\delta}} v_{0}$ be the projection onto Fourier modes up to $N_{\delta}$. 

We remark that the last bound $\|\tilde{v}\|_{L_{t}^{\infty}H^{100}_{x}(\RRR\times \RRR)}\leq C_{\delta}$ depends on $\delta$, and should be understood as a crude bound, but we will see the error caused by this bad bound will vanishes asymptotically.

Note that by Theorem \ref{th:dodson}, there exists some $C_{M_{0}}$ independent of $\delta$ such that
\begin{equation}
\|\tilde{v}\|_{L_{t}^{5}L_{x}^{10}(\RRR\times\RRR)}\leq C_{M_{0}}.
\end{equation}
Now we want to use $\tilde{v}$ as an approximate solution to $v_{\eps}$, note that one has
\begin{equation}
\begin{aligned}
i\partial_{t}\tilde{v} +\Delta \tilde{v} = \nN^{\eps}(\tilde{v})+ (\nN(\tilde{v})- \nN^{\eps}(\tilde{v})),\\
\tilde{v}(0)=\tilde{v_{0}}. 
\end{aligned}
\end{equation}
Observe that 
\begin{equation}
\begin{aligned}
&\lim_{\eps\rightarrow 0}|\|\tilde{v}_{\eps}|^{4-\eps}\tilde{v}-|\tilde{v}|^{4}v\|_{L_{t}^{1}L_{x}^{2}(\iI)}\\
\lesssim &\lim_{\eps\rightarrow 0}|\|\tilde{v}_{\eps}|^{4-\eps}\tilde{v}-|\tilde{v}|^{4}v\|_{L_{t}^{\infty}L_{x}^{2}(\iI)}\\
\lesssim &\lim_{\eps\rightarrow 0} \eps \big(1+\|v\|_{{L_{t}^{\infty}H_{x}^{100}(\RR)}}\big)^{5}\\
=&0.
\end{aligned}
\end{equation}
Thus, by Proposition~\ref{pr:sta_eps}, we have
\begin{equation}\label{444}
\lim_{\eps}\|v_{\eps}-\tilde{v}\|_{\xX_2(\iI)} \lesssim \delta. 
\end{equation}
Since $\delta$ can be chosen abitrary small, \eqref{444} and \eqref{eq: approximate11222} imply the Lemma.
\end{proof}

\subsubsection{Case 3.1: $\lim_{n\rightarrow \infty}\lambda_{1,n}=0$,$\lim_{n}\lambda_{1,n}^{\eps_{n}}=0$, $\eps_{n}$ asymptotically large}

We need only consider backforward scattering profile and compact profile. We start with the backforward scattering profile.

As mentioned in \eqref{eq: caseofinterest}, we only consider the subcase
\begin{equation}
\lim_{n\rightarrow \infty}\lambda_{1,n}=0 \text{  and } \liminf_{n\rightarrow \infty}{t_{1,n}}\leq 1.
\end{equation}
We will see in this case the dynamic is essentially linear. For convenience, we naturally extend our solution $v_{1,n}$ to $[0,2]$. We want to prove the following. 

\begin{lem}\label{lem: dynamicesslinear}
One can find $t_{0}$ in [0,2), such that
\begin{equation}\label{eq: nodynamicbefore}
\lim_{n\rightarrow \infty}\|v_{1,n}(t)-e^{it\Delta}\phi_{1,n}\|_{\xX[0,t_{0}]}=0
\end{equation}
\begin{equation}\label{eq: nodynamicafter}
\lim_{n\rightarrow \infty}\|e^{i(t+t_{0})\Delta}v_{1,n}(t_{0})\|_{\xX_2(\RR)} = 0.
\end{equation}
\end{lem}

Assuming Lemma \ref{lem: dynamicesslinear} at the moment, let us finish the proof of Lemma \ref{lem: oneprofilekey}. Note that \eqref{eq: nodynamicafter} combined with Lemma~\ref{lem: prelwp} implies
\begin{equation}
\lim_{n\rightarrow \infty}\|v_{1,n}(t+t_{0})-e^{it\Delta}v_{1,n}(t_{0})\|_{\xX([0,2-t_{0}])}=0,
\end{equation}
which, combined with \eqref{eq: nodynamicbefore}, implies
\begin{equation}
\lim_{n\rightarrow \infty}\|v_{1,n}(t)-e^{it\Delta}\phi_{1,n}\|_{\xX(\iI)}=0.
\end{equation}
Lemma \ref{lem: oneprofilekey} thus follows.

We are left with the proof of Lemma~\ref{lem: dynamicesslinear}. We will indeed choose $t_{0}=3/2$. Fixing $\kappa$ arbitrarily small, there exists some $T>0$ such that
\begin{equation}\label{eq: smallbefore}
\|e^{it\Delta}\phi_{1}\|_{\xX_2(-\infty,-T]} \leq \kappa,
\end{equation}
and 
\begin{equation}\label{eq: smallafter}
\|e^{it\Delta}\phi_{1}\|_{\xX_2([T,\infty))} \leq \kappa, 
\end{equation}
which implies
\begin{equation}\label{eq: resmallbefore}
\|e^{it\Delta}(\phi_{1,n})\|_{\xX_2[0,t_{1,n}-\lambda_{n}^{2}T]}\leq \kappa
\end{equation}
\begin{equation}\label{eq: resmallafter}
\|e^{it\Delta}(\phi_{1,n})\|_{\xX_2(t_{1,n}+\lambda_{1,n}^{2}T, 2)}\leq \kappa
\end{equation}

Recall that we have, by \eqref{eq: caseofinterest},
\begin{equation}
\lim_{n\rightarrow}t_{1,n}\pm \lambda_{1,n}^{2}T\leq 1
\end{equation}
By Lemma~\ref{lem: prelwp}, equation \eqref{eq: resmallbefore} implies that (for every large $n$)
\begin{equation}\label{eq: nonlinearsmallbefore}
\|e^{it\Delta}(\phi_{1,n})-v_{1,n}(t)\|_{\xX([0,t_{1,n}-\lambda_{n}^{2}T])}\lesssim \kappa
\end{equation}
On the other hand , the condition $\lim_{n\rightarrow 0}\lambda_{1,n}^{\eps_{n}}=0$ implies via Lemma \ref{lem: prelwp}
\begin{equation}\label{eq: nonlinearsmallmiddle}
\lim_{n\rightarrow \infty}\|v_{1,n}(t)-e^{i\big(t-(t_{1,n}-\lambda_{n}^{2}T)\big)\Delta}v_{1,n}(t_{1,n}-\lambda_{n}^{2}T)\|_{\xX([t_{1,n}-\lambda_{1,n}^{2}T,t_{1,n}+\lambda_{1,n}^{2}]T)}=0, 
\end{equation}
Basically, since we are consider the dynamic within time scale $\sim \lambda_{1,n}^{2}$ in \eqref{eq: nonlinearsmallmiddle}, the condition $\lim_{n\rightarrow 0}\lambda_{1,n}^{\eps_{n}}=0$ says the time is too short to exhibit any nonlinear phenomena  in the spirit of Lemma \ref{lem: prelwp}.

Estimate \eqref{eq: nonlinearsmallmiddle} in turn implies
\begin{itemize}
\item For any $2\geq s>t_{1,n}+\lambda_{n}^{2}T$ with $n$ large, $\|v_{1,n}(t)-e^{i(s)\Delta}(\phi_{1,n})\|_{L_{x}^{2}}\lesssim \kappa$,
\item Thus, for any $2 \geq s>t_{1,n}+\lambda_{n}^{2}T$ with $n$ large,
\begin{equation}\label{eq: nonlinearsmallafter}
\limsup_{n\rightarrow\infty}\|e^{it\Delta}v_{1,n}(s)\|_{\xX_2(\RR^+)}\lesssim \kappa+\|e^{i(t+T)}\phi_{1}\|_{\xX_2(\RR^+)} \lesssim \kappa. 
\end{equation}
\end{itemize}
Estiamtes \eqref{eq: nonlinearsmallbefore}, \eqref{eq: nonlinearsmallmiddle}, \eqref{eq: nonlinearsmallafter}, and the fact $2>3/2>\lim_{n\rightarrow \infty}t_{1,n}+\lambda_{n}^{2}T$ together imply Lemma~\ref{lem: dynamicesslinear}.

For the compact profile , one only needs an analogue of Lemma \ref{lem: dynamicesslinear}, with exactly same statement except for $v_{1,n}$ is a compact profile rather than backward scattering profile. The proof is essentially same. We left the details to readers.

\subsubsection{Case 3.2: $\lim_{n\rightarrow \infty}\lambda_{1,n}=0$, $\lim_{n}\lambda_{1,n}^{\eps_{n}}=c_{1}^{2}>0$, $\eps_{n}$ asymptotically small}

If $\liminf_{n}\lambda_{1,n}^{\eps_{n}}>0$, without loss of generality and up to a subsequence, we can assume $\lim_{n}\lambda_{1,n}^{\eps_{n}}=c_{1}^{2}>0$.

In Case 3.1, we start with the backward scattering profile. This time, we start with the compact profile. Let $w_{n}(t,x)=\lambda_{1,n}^{1/2}v_{1,n}(\lambda_{1,n}^{2}t,\lambda_{1,n}x)$, then we know $w_{n}$ solves in $[0,\frac{1}{\lambda_{n}^{2}}]$ the equation
\begin{equation}\label{eq: eqforw}
\begin{cases}
i\partial_{t}w_{n}+\Delta w_{n}=\lambda_{n}^{\frac{\eps_{n}}{2}}|w_{n}|^{4-\eps_{n}},\\
w_{n}(0,x)=\phi_{1}. 
\end{cases}
\end{equation}
Recall $\Psi_{1}=\lambda_{1,n}^{1/2}\Psi_{1,n}(\lambda_{1,n}^{2}t, \lambda_{1,n}x)$, and $\Psi_{1}$ solves
\begin{equation}
i\partial_{t}\Psi_{1}+\Delta \Psi_{1}=c_{1}|\Psi_{1}|^{4}\Psi_{1}. 
\end{equation}
Note that in some sense, it is not a good idea to directly use $\Psi_{1}$ to approximate $w_{n}$ in such a long interval $[0,\frac{1}{\lambda_{1,n}^{2}}]$. We will compare $v_{1,n}$ and $\Psi_{1,n}$ in terms of $\Psi_{1}$ and $w_{n}$ in recaled time interval for some large $T$,  and shows that after that time both $v_{1,n}$ and $\Psi_{1,n}$ evolves essentially in a linear way. We will indeed prove the following lemma. 

\begin{lem}\label{lem: auxkey}
For every $\kappa>0$, there exists some $T>0$ such that
\begin{equation}\label{eq: smallnessbefore22}
\limsup_{n\rightarrow \infty} \|\Psi_{1}-w_{n}\|_{\xX([0,T])} \lesssim \kappa
\end{equation}
and 
\begin{equation}\label{eq: smallnessafter22}
\begin{cases}
\limsupn \|e^{it\Delta}w_{n}(T)\|_{\xX_2(\RR^+)}\lesssim \kappa,\\
\limsupn \|e^{it\Delta}\Psi_{1}(T)\|_{\xX_2(\RR^+)}\lesssim \kappa. 
\end{cases}
\end{equation}
\end{lem}

Assume Lemma \ref{lem: auxkey} for the moment. Note that \eqref{eq: smallnessbefore22} implies
\begin{equation}\label{eq: cin11}
\limsup_{n\rightarrow \infty} \|\Psi_{1,n}-v_{1,n}\|_{\xX([0,\lambda_{1,n}^{2}T]}\lesssim \kappa, 
\end{equation}
and that \eqref{eq: smallnessafter22} implies
\begin{equation}
\begin{cases}
\limsupn \|e^{it\Delta}v_{1,n}(\lambda_{1,n}^{2}T)\|_{\xX_2(\RR^+)}\lesssim \kappa,\\
\limsupn \|e^{it\Delta}\Psi_{1,n}(\lambda_{1,n}^{2}T)\|_{\xX_2(\RR^+)}\lesssim \kappa,
\end{cases}
\end{equation}
which in turn implies, by Lemma \ref{lem: prelwp}, that
\begin{equation}\label{eq: cin22}
\begin{cases}
\limsupn \|v_{1,n}(t+\lambda_{1,n}^{2}T)-e^{it\Delta}v_{1,n}(\lambda_{1,n}^{2}T)\|_{\xX ([0,1-\lambda_{1,n}^{2}T))}\lesssim \kappa,\\
\limsupn \|\Psi_{1,n}(t+\lambda_{1,n}^{2}T)-e^{it\Delta}\Psi_{1,n}(\lambda_{1,n}^{2}T)\|_{\xX ([0,1-\lambda_{1,n}^{2}T))}\lesssim \kappa.
\end{cases}
\end{equation}
Clearly Lemma~\ref{lem: oneprofilekey} follows from \eqref{eq: cin11} and \eqref{eq: cin22} since $\kappa$ can be chosen arbitrarily small.

Now, we turn to the proof of Lemma~\ref{lem: auxkey}. Let $\|\phi_{1}\|_{L_{x}^{2}}=M$. By Theorem \ref{th:dodson}, there exists some $C_{M}>0$ such that
\begin{equation}
\|\Psi_{1}\|_{\xX_2(\RR^+)} \leq C_{M}. 
\end{equation}
Fixing $\kappa$ small, there exists some $T>0$ such that
\begin{equation}
\|\Psi_{1}\|_{\xX_2([T,\infty))}\lesssim \kappa
\end{equation}
which is equivalent  (via essentially the local well posedness Lemma \ref{lem: prelwp} ) to
\begin{equation}
\|e^{it\Delta}\Psi_{1}(T)\|_{\xX_2([T,\infty))} \lesssim \kappa.
\end{equation}
Lemma~\ref{lem: auxkey} follows if we can show \eqref{eq: smallnessbefore22}.

Again, in the spirit of Proposition \ref{pr:sta_eps}, we want to use $\Psi_{1}$ to approximate $w_{n}$ in time $[0,T]$ for $n$ large. First, as in case 2, we want to smoothify $\Psi_{1}$. Let $\tilde{\Psi}_{1}$ be the solution to
\begin{equation}
\begin{cases}
i\partial_{t}\tilde{\Psi}_{1}+\Delta \tilde{\Psi}_{1}=|\Psi_{1}|^{4}\tilde{\Psi}_{1},\\
\tilde{\Psi_{1}}(0)=P_{<K}\phi_{1}
\end{cases}
\end{equation}
when $K$ is large enough depending on $\kappa$, $M$. We will not track the dependence on $M$, since this $M$ is fixed all the time. One has
\begin{equation}
\begin{cases}
\|\tilde{\Psi_{1}}(0)-\phi_{1}\|_{L_{x}^{2}}\leq \kappa,\\
\|\tilde{\Psi}_{1}-\Psi_{1}\|_{\xX([0,T])}\lesssim \kappa\\
\|\tilde{\Psi}_{1}\|_{L_{t}^{\infty}H^{100}_{x}}\lesssim_{\kappa} 1
\end{cases}
\end{equation}
Note that the last bound, which is of a persistence of regularity (see for example \cite[Lemma 3.12]{colliander2008global}) is indeed a bad bound, and blows up when $\kappa$ goes to $0$. But note that $\kappa$  is already fixed now.

It is now enough to use $\tilde{\Psi}_{1}$ to approximate $w_{n}$ and prove \eqref{eq: smallnessbefore22} for $\tilde{\Psi}_{1}$. Note that $\tilde{\Psi}_{1}$ solves
\begin{equation}
i\partial_{t}\tilde{\Psi}_{1}+\Delta \tilde{\Psi}_{1}=\lambda_{n}^{\eps_{n}}|\tilde{\Psi_{1}}|^{4-\eps_{n}}\tilde{\Psi_{1}}+e_{n}, 
\end{equation}
and $e^{n}=(|\tilde{\Psi}_{1}|^{4}\|\Psi_{1}-|\tilde{\Psi}_{1}|^{4-\eps_{n}}|\tilde{\Psi_{1}}|)+\|(c_{1}-\lambda_{n}^{\eps_{n}/2})|\tilde{\Psi}_{1}|^{4-\eps_{n}}\tilde{\Psi}_{1}$. Clearly the desired estimates follows from Propostion~\ref{pr:sta_eps} if one can show
\begin{equation}
\lim_{n} \|e_{n}\|_{L_{t}^{1}L_{x}^{2}[0,T]} = 0, 
\end{equation}
which follows from
\begin{equation}
\lim_{n} \|e_{n}\|_{L_{t}^{\infty}L_{x}^{2}[0,T]} \rightarrow 0. 
\end{equation}
The last inequality is obvious since
\begin{equation}
\||\tilde{\Psi}|^{4}\|\Psi_{1}-|\tilde{\Psi}_{1}|^{4-\eps_{n}}\|_{L_{x}^{2}}\lesssim \eps_{n}\|\tilde{\Psi}_{1}\|_{L_{x}^{2}}(1+\|\tilde{\Psi_{1}}\|_{L_{x}^{\infty}})\lesssim \eps_{n}\|\tilde{\Psi}_{1}\|_{H_{x}^{100}},
\end{equation}
and 
\begin{equation}
\|(c_{j}-\lambda_{n}^{\eps_{n}/2})|\tilde{\Psi}_{1}|^{4-\eps_{n}}\tilde{\Psi}_{1}\|_{L_{x}^{2}}\lesssim (c_{j}-\lambda_{1,n}^{\eps_{n}/2})(1+\|\tilde{\Psi}_{1}\|_{H^{100}}^{5}.
\end{equation}
The subcase for the compact profile is thus proved.

The proof for the  backward scattering profile is essentially the same. and indeed be reduced back the compact profile case. We birefly present it for the convenience of the readers. The idea is to evolve the solution a little bit so that the backward scattering profile become compact profile.

We extend $v_{1,n}$ to time interval $[0,2]$ for convenience. Recall that we only consider the subcase so that \eqref{eq: caseofinterest} holds. We only need to prove the following. 

\begin{lem}\label{lem: finaltech}
For every $\kappa>0$ small enough, there exists some $T>0$, so that 
\begin{equation}
\|\Psi_{1,n}-v_{1,n}\|_{\xX((0, t_{1,n}-\lambda_{1,n}^{2}T))} \leq \kappa
\end{equation}
\end{lem}

Note that $ \Psi_{1,n}(t_{1,n}-\lambda_{1,n}^{2}T)=\frac{1}{\lambda_{1,n}^{d/2}}\Psi(T,\frac{x}{\lambda_{1,n}})$, the analysis after $ t_{1,n}-\lambda_{1,n}^{2}T$ for $v_{1,n}$ is same as the analysis of compact profile.

Lemma~\ref{lem: finaltech} follows easily from the fact for every $\kappa>0$ small, one will be able to find $T>0$ such that
\begin{equation}\label{eq:12345}
\limsupn \|\Phi_{1}\|_{\xX_2([-t_{1,n}/\lambda_{1,n}^{2}, -T))}\leq \|\Phi_{1}\|_{\xX_2([-\infty, -T))}\lesssim \kappa. 
\end{equation}
Indeed, \eqref{eq:12345} implies 
\begin{equation}\label{eq: 23456}
\limsupn\|e^{it\Delta}\phi_{1,n}\|_{\xX_2([0,t_{1,n}-\lambda_{1,n}^{2}T]}\lesssim \kappa
\end{equation}
and since $v_{1,n}$ and $\Psi_{1,n}$ has same initial data. The desired estimate follows from \eqref{eq: 23456} by Lemma~\ref{lem: prelwp}.

\section{Acknowledgment}
We thank Carlos Kenig for helpful discussions.

\appendix

\section{Proof of Proposition~\ref{pr:sta_me}}\label{secproofofprstame}

\begin{proof}
	[Proof of Proposition~\ref{pr:sta_me}]
	We fix a sufficiently small number $\eta>0$. Its value, depending on $M_1$ and $M_2$ only, will be specified later. Since $\|v\|_{\xX_2(\iI)}^{5} \leq M_2^5$, we can choose a dissection $\{\tau_k\}_{k=0}^{K}$ of the interval $\iI = [a,b]$ as follows. Let $\tau_0 = a$. Suppose $\tau_k$ is chosen, we determine $\tau_{k+1}$ by
	\begin{equation*}
	\tau_{k+1} = b \wedge \inf \big\{\tau>\tau_{k}: \|v\|_{\xX_2(\tau_k,\tau)}^{5} = \eta \big\}, 
	\end{equation*}
	and the process necessarily stops at $\tau_K = b$ with
	\begin{equation*}
	K \leq 1 + \frac{M_2^5}{\eta}. 
	\end{equation*}
	Now, for every $k \leq K-1$ and every $t \in [\tau_k, \tau_{k+1}]$, we have
	\begin{equation} \label{eq:duhamel_weak_sta}
	\begin{split}
	v(t) - w(t) = &e^{i (t-\tau_k) \Delta} \big( v(\tau_k) - w(\tau_k) \big) - \int_{\tau_k}^{t} \sS(t-s) \gG(s) {\rm d}s\\
	&- i \int_{\tau_k}^{t} \sS(t-s) e(s) {\rm d}s, 
	\end{split}
	\end{equation}
	where
	\begin{equation*}
	\gG(s) =  \theta_{m}\big(\tilde{A}+\|v\|_{\xX_2(a,s)}^{5}\big) \nN^{\eps} \big(v(s)\big) - \theta_{m}\big(A+\|w\|_{\xX_2(a,s)}^{5}\big) \nN^{\eps} \big(w(s)\big). 
	\end{equation*}
	We first give a pointwise control of the integrand inside the operator $\sS(t-s)$ on the second line above. It can be split into two parts $\gG = \gG_1 + \gG_2$ such that
	\begin{equation*}
	\begin{split}
	\gG_1(s) &= \theta_{m}\big(\tilde{A}+\|v\|_{\xX_2(a,s)}^{5}\big) \nN^{\eps} \big(v(s)\big) - \theta_{m}\big(A+\|w\|_{\xX_2(a,s)}^{5}\big) \nN^{\eps} \big(v(s)\big)\\
	\gG_2(s) &= \theta_{m}\big(A+\|w\|_{\xX_2(a,s)}^{5}\big) \nN^{\eps} \big(v(s)\big) - \theta_{m}\big(A+\|w\|_{\xX_2(a,s)}^{5}\big) \nN^{\eps} \big(w(s)\big). 
	\end{split}
	\end{equation*}
	Thus, for every $\tau_k \leq s \leq r \leq \tau_{k+1}$, we have
	\begin{equation*}
	\begin{split}
	|\gG_1(s)| &\leq C |v(s)|^{5-\eps} \Big( |\tilde{A}-A| + \|v-w\|_{\xX_2(a,r)} \big( M_{2}^{4} + \|v-w\|_{\xX_2(a,r)}^{4} \big) \Big)\\
	|\gG_2(s)| &\leq C |v(s)-w(s)| \big( |v(s)|^{4-\eps} + |v(s)-w(s)|^{4-\eps} \big). 
	\end{split}
	\end{equation*}
	Hence, by Strichartz estimates and H\"older, we have
	\begin{equation*}
	\begin{split}
	&\Big\| \int_{\tau_k}^{t} \sS(t-s) \gG(s) {\rm d}s \Big\|_{\xX(\tau_k,r)} \leq  C \|v-w\|_{\xX(\tau_k,r)} \big( \|v\|_{\xX_2(\tau_k,r)}^{4-\eps} + \|v-w\|_{\xX(\tau_k,r)}^{4-\eps} \big)\\
	&+C \|v\|_{\xX_1(\tau_k,r)}^{\frac{\eps}{4}} \|v\|_{\xX_2(\tau_k,r)}^{5-\frac{5 \eps}{4}} \Big( |\tilde{A}-A| + \|v-w\|_{\xX_2(a,r)} \big( M_{2}^{4} + \|v-w\|_{\xX_2(a,r)}^{4} \big) \Big). 
	\end{split}
	\end{equation*}
	Now, plugging it back into \eqref{eq:duhamel_weak_sta} and using $\|v\|_{\xX_2(\tau_k, \tau_{k+1})} \leq \eta$, we get
	\begin{equation*}
	\begin{split}
	\|v-w\|_{\xX(\tau_k,r)} &\leq C \Big( \|v(\tau_k)-w(\tau_k)\|_{L_{x}^{2}} + \|e\|_{L_{t}^{1}L_{x}^{2}(\iI)} +  \eta^{4-\eps} \|v-w\|_{\xX(\tau_k,r)} + \|v-w\|_{\xX_2(\tau_k,r)}^{5-\eps} \Big)\\
	&+ C M_{1}^{\frac{\eps}{4}} \eta^{5-\frac{5\eps}{4}} \Big( |\tilde{A}-A| + M_{2}^{4} \|v-w\|_{\xX(a,r)} + \|v-w\|_{\xX(a,r)}^{5} \Big). 
	\end{split}
	\end{equation*}
	Now, we choose $\eta$ sufficiently small such that
	\begin{equation} \label{eq:choice_eta}
	C \eta^{4-\eps} < \frac{1}{4} \quad \text{and} \quad C (M_1 + M_2)^{4+\frac{\eps}{4}} \eta^{5-\frac{5\eps}{4}} < \frac{1}{4}
	\end{equation}
	for all $\eps \in [0,1]$. It is clear that this choice of $\eta$ depends on $M_1$ and $M_2$ only. With this choice, we can move and merge the terms proportional to $\|v-w\|_{\xX(\tau_k,r)}$ to the left hand side. Then, adding $\|v-w\|_{\xX(a,r)}$ on both sides, we get
	\begin{equation} \label{eq:sta_me_bootstrap}
	\|v-w\|_{\xX(a,r)} \leq C_0 \Big( \|v-w\|_{\xX(a,\tau_k)} + \|e\|_{L_{t}^{1}L_{x}^{2}(\iI)} + |\tilde{A}-A| + \|v-w\|_{\xX(a,r)}^{5-\eps} + \|v-w\|_{\xX(a,r)}^{5} \Big)
	\end{equation}
	for all $r \in [\tau_k, \tau_{k+1}]$, and $C_0$ is universal. Note that the constants in front of $|\tilde{A} - A|$ and $\|v-w\|_{\xX(a,r)}^{5}$ do not depend on $M_1$ since the choice of $\eta$ balances it out. 
	
	Let
	\begin{equation*}
	\delta_{k} = \|v-w\|_{\xX(a,\tau_k)} + \|e\|_{L_{t}^{1}L_{x}^{2}(\iI)} + |\tilde{A} - A|. 
	\end{equation*}
	According to \eqref{eq:sta_me_bootstrap} and the standard continuity argument, there exists $\delta^*$ universal such that if $\delta_k < \delta^*$, then
	\begin{equation*}
	\|v-w\|_{\xX(a,\tau_{k+1})} \leq 2 C_0 \delta_k, 
	\end{equation*}
	and consequently
	\begin{equation} \label{eq:iterate_diff}
	\delta_{k+1} \leq (2C_0 + 1) \delta_k. 
	\end{equation}
	If the initial difference $\delta_0$ is small enough such that $\delta_k < \delta^*$ for all $k \leq K$, we can then iterate \eqref{eq:iterate_diff} up to $K$ so that
	\begin{equation*}
	\delta_K \leq (2C_0+1)^{K} \delta_0. 
	\end{equation*}
	Note that the iteration hypothesis $\delta_k < \delta^*$ is satisfied for all $k \leq K$ if
	\begin{equation*}
	\delta_0 < \frac{\delta^*}{(2C_0+1)^{1 + M_2^5/\eta}}. 
	\end{equation*}
	Note that this choice of $\delta_0$ depends on $M_1$ and $M_2$ only since $\eta$ does. Since $K \leq 1 + M_2^5/\eta$, it guarantees that the iteration hypothesis is satisfied for $K$ and hence
	\begin{equation*}
	\|v-w\|_{\xX(\iI)} \leq \delta_K \leq (2C_0+1)^{1+\frac{M_2^5}{\eta}} \delta_0. 
	\end{equation*}
	This completes the proof of the proposition by the definition of $\delta_0$ and by taking $C = C_{M_1, M_2} = (2C_0+1)^{1+M_2^5/\eta}$, where $\eta$ is chosen according to \eqref{eq:choice_eta}. 
\end{proof}

\section{Propositions~\ref{pr:sta_me} +~\ref{pr:bd_me} imply Proposition~\ref{pr:strong_sta_me}}\label{secnaturalargu}

Proposition~\ref{pr:bd_me} gives the uniform boundedness of $\|w\|_{\xX(\iI)}$. As a consequence, we can enhance the stability statement by dropping the assumption on $\|v\|_{\xX_2(\iI)}$. Also, since $\|w\|_{L_x^2}$ is conserved, we only need the assumption on the $L_x^2$ norm of the initial data rather than the whole interval. 

\begin{proof} [Proof of Proposition~\ref{pr:strong_sta_me}]
	The proof is essentially the same as that for Proposition~\ref{pr:sta_me} except that thanks to Proposition~\ref{pr:bd_me}, instead of the assumptions on $\|v\|_{\xX_1(\iI)}$ and $\|v\|_{\xX_2(\iI)}$, we now have
	\begin{equation*}
	\|w\|_{\xX_1(\iI)} \leq M\;, \qquad \|w\|_{\xX_2(\iI)} \leq D_M. 
	\end{equation*}
	as a fact rather than assumption. Hence, we determine the dissection $a = \tau_0 < \cdots < \tau_K = b$ by
	\begin{equation*}
	\tau_{k+1} = b \wedge \inf \big\{ \tau>\tau_k: \|w\|_{\xX_2(\tau_k,\tau)}^{5} = \eta \big\}. 
	\end{equation*}
	By choosing $\eta$ sufficiently small depending on $M$ and $D_M$ only, we can get a bound of the same form as \eqref{eq:sta_me_bootstrap}. The claim then follows in the same way as above. 
\end{proof}

\section{Burkholder inequality and proof of Proposition~\ref{pr:bound_maximal}}

Let $W = \Phi \tilde{W}$ as in Assumption~\ref{as:noise}. Burkholder inequality (\cite{BDG}, \cite{Burkholder}) will be essential for our analysis. We will use the following version (\cite{BP}, \cite{Brzezniak}, \cite{UMD}). 

\begin{prop} \label{pr:Burkholder}
	Let $(\fF_t)_{t \geq 0}$ be the filtration generated by $W$ and $\sigma$ be a process adapted to $(\fF_t)_{t \geq 0}$. Then for every $p \in [2,+\infty)$ and $\rho \in [1,+\infty)$, there exists $C>0$ depending on $p$, $\rho$ and $T$ only such that
	\begin{equation} \label{eq:Burkholder}
	\EE \sup_{t \in [0,T]} \Big\| \int_{0}^{t} \sigma(s) {\rm d} W_s \Big\|_{L^{p}}^{\rho} \leq C \EE \Big( \int_{0}^{T} \|\sigma(s) \Phi\|_{\R(L^2, L^p)}^{2} {\rm d}s \Big)^{\frac{\rho}{2}}. 
	\end{equation}
	Here, the operator $\sigma(s) \Phi$ is the action of $\Phi$ followed by the multiplication of $\sigma(s)$. 
\end{prop}

\begin{proof} [Proof of Proposition~\ref{pr:bound_maximal}]
	Since $\rho \geq 5$, we can use Minkowski inequality to change of order of integration and then apply H\"older to get
	\begin{equation*}
	\|M^*\|_{L_{\omega}^{\rho}L_{t}^{5}(0,T_0)} \leq \|M^*\|_{L_{t}^{5}((0,T_0),L_{\omega}^{\rho})} \leq T_{0}^{\frac{1}{5}} \sup_{t \in [0,T_0]} \|M^*(t)\|_{L_{\omega}^{\rho}}. 
	\end{equation*}
	By definition of $M^*$, we have
	\begin{equation*}
	|M^*(t)| \leq 2 \sup_{0 \leq \tau \leq t} \Big\|\int_{0}^{\tau}\sS(t-s) u_{m,\eps}(s) {\rm d}s \Big\|_{L_{x}^{10}}, 
	\end{equation*}
	so we can apply \eqref{eq:Burkholder} to obtain
	\begin{equation*}
	\|M^*(t)\|_{L_{\omega}^{\rho}} \leq C_{\rho} \EE \Big( \int_{0}^{t} \|\sS(t-s) u_{m,\eps}(s)\|_{\R(L_{x}^{2}, L_{x}^{10})}^{2} {\rm d}s \Big)^{\frac{\rho}{2}}. 
	\end{equation*}
	Now, apply Proposition~\ref{pr:dispersive} to the integrand above and use factorization, we get
	\begin{equation*}
	\|M^*(t)\|_{L_{\omega}^{\rho}}^{\rho} \leq C_{\rho} T_{0}^{\frac{\rho}{10}} \|\Phi\|_{\R(L_{x}^{2}, L_{x}^{5/2})} \EE \|u_{m,\eps}\|_{\xX_1(0,T_0)}^{\rho}. 
	\end{equation*}
	Note that the right hand side above does not depend on $t$, so the inequality also holds with the left hand side substituted by $\sup_{t} \|M^*(t)\|_{L_{\omega}^{\rho}}$. The claim then follows since $\|u_{m,\eps}\|_{L_{\omega}^{\rho}} \leq M$ by pathwise mass conservation. 
\end{proof}

\endappendix

\bibliographystyle{Martin}
\bibliography{Refs}

\end{document}